\documentclass[oneside,smallextended]{amsart}
\usepackage[T1]{fontenc}
\usepackage[latin9]{inputenc}
\usepackage{rotating}
\usepackage{amsthm}
\usepackage{amstext}
\usepackage{amssymb}
\usepackage{graphicx}

\makeatletter

\providecommand{\tabularnewline}{\\}

\numberwithin{equation}{section}
\numberwithin{figure}{section}
\theoremstyle{plain}
\newtheorem{thm}{\protect\theoremname}
  \theoremstyle{plain}
  \newtheorem{prop}[thm]{\protect\propositionname}
  \theoremstyle{definition}
  \newtheorem{defn}[thm]{\protect\definitionname}
  \theoremstyle{remark}
  \newtheorem{rem}[thm]{\protect\remarkname}
  \theoremstyle{definition}
  \newtheorem{example}[thm]{\protect\examplename}
  \theoremstyle{plain}
  \newtheorem{cor}[thm]{\protect\corollaryname}
\newenvironment{lyxlist}[1]
{\begin{list}{}
{\settowidth{\labelwidth}{#1}
 \setlength{\leftmargin}{\labelwidth}
 \addtolength{\leftmargin}{\labelsep}
 }}
{\end{list}}

\usepackage{perry}
\usepackage{amsfonts}
\usepackage{amsmath}

\newcommand\posints{\ensuremath{\mathbb N^+}}

\newcommand\spoly{\ensuremath{\mathrm{spoly}}}
\newcommand\redmod[1]{\ensuremath{\underset{#1}{\longrightarrow}}}
\newcommand\supp{\ensuremath{\mathrm{supp}}}

\newcommand\allorders{\ensuremath{\mathcal T}}

\newcommand\lp{\ensuremath{\mathrm{lp}}}
\newcommand\mlp{\ensuremath{\mathrm{mlp}}}
\newcommand\rejects{\ensuremath{\mathit{rejects}}}

\newcommand\hilfun{\ensuremath{H}}

\renewcommand\monomials{\ensuremath{\mathbb{T}^n}}
\renewcommand\field{K}

\makeatother

  \providecommand{\corollaryname}{Corollary}
  \providecommand{\definitionname}{Definition}
  \providecommand{\examplename}{Example}
  \providecommand{\propositionname}{Proposition}
  \providecommand{\remarkname}{Remark}
\providecommand{\theoremname}{Theorem}

\begin{document}

\title[Dynamic Gr\"obner basis algorithms]{Reducing the size and number of linear programs in a dynamic Gr\"obner
basis algorithm}

\author{Massimo Caboara}

\address{Universit\'a di Pisa}

\email{caboara@dm.unipi.it}

\author{John Perry}

\address{University of Southern Mississippi}

\email{john.perry@usm.edu}

\urladdr{www.math.usm.edu/perry/}

\begin{abstract}
The dynamic algorithm to compute a Gr\"obner basis is nearly twenty
years old, yet it seems to have arrived stillborn; aside from two
initial publications, there have been no published followups. One
reason for this may be that, at first glance, the added overhead seems
to outweigh the benefit; the algorithm must solve many linear programs
with many linear constraints. This paper describes two methods that
reduce both the size and number of these linear programs.
\end{abstract}
\maketitle

\section{Introduction}

Since the first algorithm to compute Gr\"obner bases was described
by \cite{Buchberger65}, they have become a standard tool for applied,
computational, and theoretical algebra. Their power and promise has
stimulated a half-century of research into computing them efficiently.
Important advances have resulted from reducing the number of pairs
considered~\cite{Buchberger65}\cite{Buchberger79}\cite{CKR02}\cite{Fau02Corrected}\cite{MoraMoeller86},
improving the reduction algorithm~\cite{BrickSlimgb_Journal}\cite{Fau99}\cite{YanGeobuckets},
and forbidding some reductions~\cite{ApelHemmecke05}\cite{Fau02Corrected}\cite{ZharkovBlinkov96}.

The Gr\"obner basis property depends on the choice of term ordering:
a basis can be Gr\"obner with respect to one term ordering, but not
to another. Researchers have studied ways to find an ordering that
efficiently produces a basis ``suitable'' for a particular problem~\cite{Tran2004}\cite{Tran2007},
and to convert a basis that is Gr\"obner with respect to one ordering
to a basis that is Gr\"obner with respect to another~\cite{CDR_HDriven}\cite{CKM97GBWalk}\cite{FGLM93}\cite{Traverso_HDriven}.

Another approach would be to begin without \emph{any} ordering, but
to compute both a basis \emph{and} an ordering for which that basis
is Gr\"obner. Such a ``dynamic'' algorithm would change its ordering
the moment it detected a ``more efficient'' path towards a Gr\"obner
basis, and would hopefully conclude with a smaller basis more quickly.

Indeed, this question \emph{was} posed nearly twenty years ago, and
was studied both theoretically and practically in two separate papers~\cite{CaboaraDynAlg}\cite{GS93}.
The first considered primarily questions of discrete geometry; the
dynamic algorithm spills out as a nice application of the ideas, but
the authors did not seriously consider an implementation. The second
described a study implementation that refines the ordering using techniques
from linear programming, and focused on the related algebraic questions.

Aside from one preprint~\cite{GolubitskyDynamicGB}, there has been
no continuation of this effort. There certainly are avenues for study;
for example, this observation at the conclusion of \cite{CaboaraDynAlg}:
\begin{quote}
In a number of cases, after some point is reached in the refining
of the current order, further refining is useless, even damaging.
\ldots{}The exact determination of this point is not easy, and an
algorithm for its determination is not given.
\end{quote}
An example of the damage that can occur is that the number and size
of the linear programs grow too large. The temporary solution of~\cite{CaboaraDynAlg}
was to switch the refiner off at a predetermined point. Aside from
the obvious drawback that this prevents useful refinement after this
point, it also forces unnecessary refinement before it! This can be
especially costly when working with systems rich in dense polynomials.

This paper presents two new criteria that signal the refiner not only
to switch off when it is clearly not needed, but also to switch back
on when there is a high probability of useful refinement. The criteria
are based on simple geometric insights related to linear programming.
The practical consequence is that these techniques reduce both the
number and the size of the associated linear programs by significant
proportions.

\section{Background}

This section lays the groundwork for what is to follow, in terms of
both terminology and notation. We have tried to follow the vocabulary
and notation of \cite{CaboaraDynAlg}, with some modifications.

Section~\ref{sub: Grbner bases} reviews the traditional theory of
Gr\"obner bases, inasmuch as it pertains to the (static) Buchberger
algorithm. Section~\ref{sub: dynamic algorithm} describes the motivation
and background of the dynamic algorithm, while Section~\ref{sub: Caboara's implementation}
reviews Caboara's specification. Section~\ref{sub: geometric view of dynamic algorithm}
deals with some geometric considerations which will prove useful later.

\subsection{\label{sub: Grbner bases}Gr\"obner bases and the static Buchberger
algorithm}

Let $m,n\in\posints$, $\field$ a field, and $R=\pring[n]$. We typically
denote polynomials by $f$, $g$, $h$, $p$, $q$, and $r$, and
the ideal of $R$ generated by any $F\subseteq R$ as $\ideal{F}$.
Following~\cite{CaboaraDynAlg}, we call a product of powers of the
variables of $R$ a \textbf{term}, and a product of a term and an
element of $\field$ a \textbf{monomial}. We typically denote constants
by letters at the beginning of the alphabet, and terms by $t$, $u$,
$v$. We denote the exponent vector of a term by its name in boldface;
so, if $n=4$ and $t=x_{1}^{2}x_{3}x_{4}^{20}$, then $\mathbf{t}=\left(2,0,1,20\right)$.

An ordering $\sigma$ on the set $\monomials$ of all terms of $R$
is \textbf{admissible} if it is a well-ordering that is compatible
with divisibility and multiplication; by ``compatible with divisibility,''
we mean that $t\mid u$ and $t\neq u$ implies that $t<_{\sigma}u$,
and by ``compatible with multiplication,'' we mean that $t<_{\sigma}u$
implies that $tv<_{\sigma}uv$. We consider only admissible orderings,
so henceforth we omit the qualification.

We write $\allorders$ for the set of all term orderings, and denote
orderings by Greek letters $\mu$, $\sigma$, and $\tau$. For any
$p\in R$ we write $\lt[\sigma]\left(p\right)$ and $\lc[\sigma]\left(p\right)$
for the leading term and leading coefficient of $p$ with respect
with $\sigma$. If the value of $\sigma$ is clear from context or
does not matter, we simply write $\lt\left(p\right)$ and $\lc\left(p\right)$.
For any $F\subseteq R$, we write $\lt[\sigma]\left(F\right)=\left\{ \lt[\sigma]\left(f\right):f\in F\right\} $.

Let $I$ be an ideal of $R$, and $G\subseteq I$. If for every $p\in I$
there exists $g\in G$ such that $\lt\left(g\right)\mid\lt\left(p\right)$,
then we say that $G$ is a \textbf{Gr\"obner basis of }$I$. This
property depends on the ordering; if $\sigma,\tau\in\allorders$,
it is quite possible for $G$ to be a Gr\"obner basis with respect
to $\sigma$, but not with respect to $\tau$.

It is well known that every polynomial ideal has a finite Gr\"obner
basis, regardless of the choice of ordering. Actually \emph{computing}
a Gr\"obner basis requires a few more concepts. Let $f,p,r\in R$.
We say that $p$ \textbf{reduces to }r \textbf{modulo} $f$, and write
$p\redmod{f}r$, if there exist $a\in\field$ and $t\in\monomials$
such that $p-atf=r$ and $\lt\left(r\right)<\lt\left(p\right)$. Similarly,
we say that $p$ \textbf{reduces to} $r$ \textbf{modulo} $G$, and
write
\[
p\redmod{G}r,
\]
if there exist $\left\{ i_{1},\ldots,i_{\ell}\right\} \subseteq\left\{ 1,\ldots,\#G\right\} $
such that $p\redmod{g_{i_{1}}}r_{1}$, $r_{1}\redmod{g_{i_{2}}}r_{2}$,
\ldots{}, $r_{\ell-1}\redmod{g_{i_{\ell}}}r_{\ell}=r$. If there
no longer exists $g\in G$ such that $\lt\left(g\right)$ divides
a $\lt\left(r\right)$, we call $r$ a \textbf{remainder of $p$ modulo
$G$}.
\begin{prop}[Buchberger's Characterization, \cite{Buchberger65}]
\label{prop: buchberger's characterization}$G$ is a Gr\"obner
basis of $I$ if and only if the \textbf{\emph{S}}\textbf{-polynomial}
of every $f,g\in I\backslash\left\{ 0\right\} $, or
\[
\spoly\left(f,g\right)=\lc\left(g\right)\cdot\frac{\lcm\left(\lt\left(f\right),\lt\left(g\right)\right)}{\lt\left(f\right)}\cdot f-\lc\left(f\right)\cdot\frac{\lcm\left(\lt\left(f\right),\lt\left(g\right)\right)}{\lt\left(g\right)}\cdot g,
\]
reduces to zero modulo $G$.
\end{prop}
For convenience, we extend the definition of an $S$-poly\-nomial
to allow for $0$:
\begin{defn}
Let $p\in R$. The \textbf{\emph{S}}\textbf{-polynomial of $p$ and
$0$} is $p$.
\end{defn}
Buchberger's Characterization of a Gr\"obner basis leads naturally
to the classical, \textbf{static Buchberger algorithm} to compute
a Gr\"obner basis; see Algorithm~\ref{alg: Static Algorithm}, which
terminates on account of the Hilbert Basis Theorem (applied to $\ideal{\lt\left(G\right)}$).
There are a number of ambiguities in this algorithm: the strategy
for selecting pairs $\left(p,q\right)\in P$, for instance, or how
precisely to reduce the $S$-poly\-nomials. These questions have
been considered elsewhere, and the interested reader can consult the
references cited in the introduction.
\begin{figure}
\rule[0.5ex]{1\columnwidth}{1pt}

\begin{raggedright}
\textbf{algorithm} \emph{static\_buchberger\_algorithm}
\par\end{raggedright}

\begin{raggedright}
\textbf{inputs:}
\par\end{raggedright}
\begin{itemize}
\item $F\subseteq R$
\item $\sigma\in\allorders$
\end{itemize}
\begin{raggedright}
\textbf{outputs:} $G\subseteq R$, a Gr\"obner basis of $\ideal{F}$
with respect to $\sigma$
\par\end{raggedright}

\begin{raggedright}
\textbf{do:}
\par\end{raggedright}
\begin{enumerate}
\item Let $G=\left\{ \right\} $, $P=\left\{ \left(f,0\right):f\in F\right\} $
\item \textbf{while} $P\neq\emptyset$

\begin{enumerate}
\item Select $\left(p,q\right)\in P$ and remove it
\item Let $r$ be a remainder of $\spoly\left(p,q\right)$ modulo $G$
\item \textbf{if} $r\neq0$

\begin{enumerate}
\item Add $\left(g,r\right)$ to $P$ for each $g\in G$
\item Add $r$ to $G$
\end{enumerate}
\end{enumerate}
\item \textbf{return} $G$
\end{enumerate}
\caption{\label{alg: Static Algorithm}The traditional Buchberger algorithm}

\rule[0.5ex]{1\columnwidth}{1pt}
\end{figure}

\begin{rem}
Algorithm~\ref{alg: Static Algorithm} deviates from the usual presentation
of Buchberger's algorithm by considering $S$-poly\-nomials of the
inputs with 0, rather than with each other. This approach accommodates
the common optimization of interreducing the inputs.
\end{rem}

\subsection{\label{sub: dynamic algorithm}The dynamic algorithm}

Every admissible ordering can be described using a real matrix $M$~\cite{Robbiano86}.
Terms $u,v$ are compared by comparing lexicographically $M\mathbf{u}$
and $M\mathbf{v}$. Two well-known orders are \textbf{lex} and \textbf{grevlex};
the former can be represented by an identity matrix, and the latter
by an upper-triangular matrix whose non-zero elements are identical.
\begin{example}
Consider the well-known Cyclic-4 system,
\begin{align*}
F & =\left(x_{1}+x_{2}+x_{3}+x_{4},x_{1}x_{2}+x_{2}x_{3}+x_{3}x_{4}+x_{4}x_{1},\right.\\
 & \phantom{=(x_{1}}x_{1}x_{2}x_{3}+x_{2}x_{3}x_{4}+x_{3}x_{4}x_{1}+x_{4}x_{1}x_{2},\\
 & \left.\phantom{=(x_{1}}x_{1}x_{2}x_{3}x_{4}-1\right).
\end{align*}

\begin{enumerate}
\item If we compute a Gr\"obner basis of $\ideal{F}$ with respect to lex,
we obtain a Gr\"obner basis with 6 polynomials made up of 18 distinct
terms.
\item If we compute a Gr\"obner basis of $\ideal{F}$ with respect to grevlex,
we obtain a Gr\"obner basis with 7 polynomials made up of 24 distinct
terms.
\item If we compute a Gr\"obner basis of $\ideal{F}$ according to the
matrix ordering
\[
\left(\begin{array}{cccc}
1 & 3 & 2 & 4\\
1 & 1 & 1 & 0\\
1 & 1 & 0 & 0\\
1 & 0 & 0 & 0
\end{array}\right),
\]
we obtain a Gr\"obner basis with 5 polynomials and 19 distinct terms.
\end{enumerate}
\end{example}
We can order any finite set of terms using a weight vector in $\mathbb{N}^{n}$,
and if necessary, we can extend a weight vector to an admissible ordering
by adding $n-1$ linearly independent rows. In the example above,
we extended the weight vector $\left(1\;3\;2\;4\right)$ by adding
three more rows.

The goal of the dynamic algorithm is to discover a ``good'' ordering
for given input polynomials during the Gr\"obner basis computation.
Algorithm~\ref{alg: Dynamic Buchberger Algorithm} describes a \textbf{dynamic
Buchberger algorithm}. As with the static algorithm, its basic form
contains a number of unresolved ambiguities:
\begin{figure}
\rule[0.5ex]{1\columnwidth}{1pt}

\begin{raggedright}
\textbf{algorithm} \emph{dynamic\_buchberger\_algorithm}
\par\end{raggedright}

\begin{raggedright}
\textbf{inputs:} $F\subseteq R$
\par\end{raggedright}

\begin{raggedright}
\textbf{outputs:} $G\subseteq R$ and $\sigma\in\allorders$ such
that $G$ is a Gr\"obner basis of $\ideal{F}$ with respect to $\sigma$
\par\end{raggedright}

\begin{raggedright}
\textbf{do:}
\par\end{raggedright}
\begin{enumerate}
\item Let $G=\left\{ \right\} $, $P=\left\{ \left(f,0\right):f\in F\right\} $,
$\sigma\in\allorders$
\item \textbf{while} $P\neq\emptyset$\label{enu: while loop of dynamic algorithm}

\begin{enumerate}
\item Select $\left(p,q\right)\in P$ and remove it
\item \label{enu: r full reduction of s(p,q) (dynamic)}Let $r$ be a remainder
of $\spoly\left(p,q\right)$ modulo $G$
\item \textbf{if} $r\neq0$

\begin{enumerate}
\item Add $\left(g,r\right)$ to $P$ for each $g\in G$
\item Add $r$ to $G$
\end{enumerate}
\item \label{enu: Select a new order}Select $\tau\in\allorders$
\item \label{enu: modify P after new order}Let $P=P\cup\left\{ \left(p,q\right):p,q\in G,\; p\neq q,\;\lt[\sigma]\left(p\right)\neq\lt[\tau]\left(p\right)\right\} $
\item Let $\sigma=\tau$
\end{enumerate}
\item \textbf{return} $G$, $\sigma$
\end{enumerate}
\caption{\label{alg: Dynamic Buchberger Algorithm}A basic, dynamic Buchberger
algorithm}

\rule[0.5ex]{1\columnwidth}{1pt}
\end{figure}

\begin{itemize}
\item How shall we select a pair?
\item How do we determine a good ordering?
\item How do we choose the ordering?
\item When should we select a new ordering?
\item Does the algorithm actually terminate?
\end{itemize}

\subsection{\label{sub: Caboara's implementation}Caboara's implementation}

The only implementation of a dynamic algorithm up to this point was
that of~\cite{CaboaraDynAlg}; it has since been lost. This section
reviews that work, adapting the original notation to our own, though
any differences are quite minor.

\subsubsection{How shall we select a pair?}

Caboara used the \textbf{sugar strategy}, which selects $\left(p,q\right)\in P$
such that the degree of the homogenization of $\spoly\left(p,q\right)$
is minimal~\cite{BigattiCaboaraRobbiano10}\cite{SugarStrategy91}.
Pairs were pruned using the Gebauer-M\"oller algorithm~\cite{GM88}.

\subsubsection{How do we determine a good ordering?}

Both \cite{GS93,CaboaraDynAlg} suggest using the \textbf{Hilbert-Poincar\'e
function} to evaluate orderings. Roughly speaking, the Hilbert-Poincar\'e
function of an ideal $I$ in a ring $R$, denoted $\hilfun_{R/I}\left(d\right)$,
indicates:
\begin{itemize}
\item if $I$ is inhomogeneous, the number of elements in $R/I$ of degree
no greater than $d$;
\item if $I$ is homogeneous, the number of elements in $R/I$ of degree
$d$.
\end{itemize}
The homogeneous case gives us the useful formula
\[
H_{R/I}\left(d\right)=\dim_{\field}\left(R/I\right)_{d},
\]
where the dimension is of the subspace of degree-$d$ terms of the
vector space $R/I$. If $G$ is a Gr\"obner basis, then $H_{R/\ideal{G}}=H_{R/\ideal{\lt\left(G\right)}}$,
and it is easy to compute the related \textbf{Hilbert series} or \textbf{Hilbert}
\textbf{polynomial} for $\hilfun_{R/\ideal{G}}\left(d\right)$ from
$\ideal{\lt\left(G\right)}$ \cite{BayerStillman92}\cite{Bigatti97}\cite{RouneHilbert2010}.
Many textbooks, such as~\cite{KR05}, contain further details.

In the homogeneous setting, the Hilbert function is an invariant of
the ideal regardless of the ordering. Thus, we can use it to measure
``closeness'' of a basis to a Gr\"obner basis. Both the static
and dynamic algorithms add a polynomial $r$ to the basis $G$ if
and only if $\ideal{\lt\left(G\right)}\subsetneq\ideal{\lt\left(G\cup\left\{ r\right\} \right)}$.
(See Figure~\ref{fig: Adding r to G makes R/<lm(G)> smaller}.)
\begin{figure}
\begin{centering}
\includegraphics[scale=0.4]{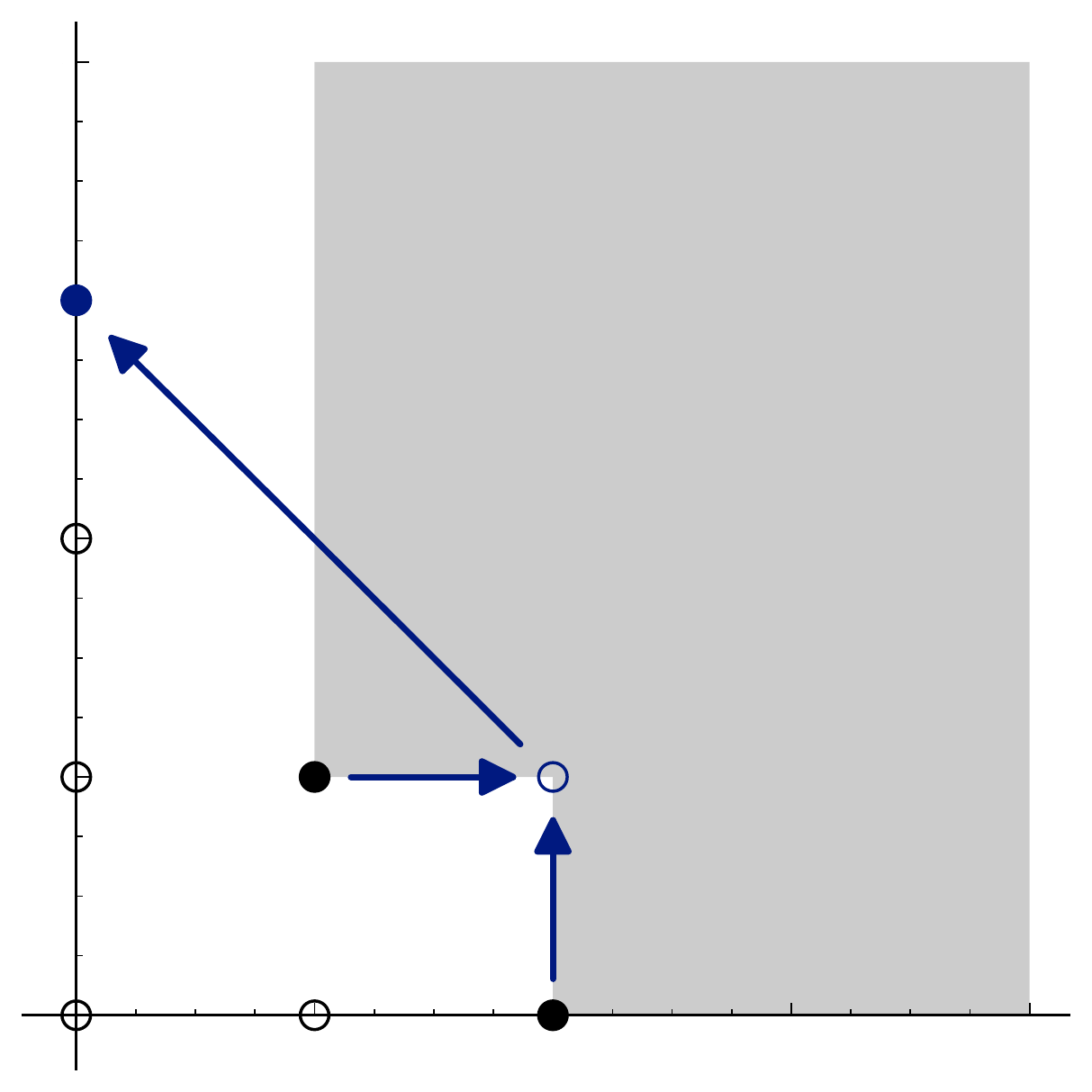}
\par\end{centering}

\caption{\label{fig: Adding r to G makes R/<lm(G)> smaller}Adding $r$ to
$G$ makes $R/\ideal{\lt\left(G\right)}$, and therefore $H_{\ideal{\lt\left(G\right)}}$,
smaller. The example here is taken from $G=\left\{ x^{2}+y^{2}-4,xy-1\right\} $
with $x>y$; running the Buchberger algorithm on the pair $\left(g_{1},g_{2}\right)$
gives us $\lt\left(r\right)=y^{3}$, which removes all multiples of
$y^{3}+\ideal{G}$ from $R/\ideal{G}$, thereby decreasing $H_{\ideal{\lt\left(G\right)}}$.}
\end{figure}
 If we denote $T=\ideal{\lt\left(G\right)}$ and $U=\ideal{\lt\left(G\cup\left\{ r\right\} \right)}$,
then $T\subsetneq U$, so for all $d$,
\[
\hilfun_{R/T}\left(d\right)=\dim_{\field}\left(R/T\right)_{d}\geq\dim_{\field}\left(R/U\right)_{d}=\hilfun_{R/U}\left(d\right).
\]
If the choice of ordering means that we have two possible values for
$U$, we should aim for the ordering whose Hilbert function is \emph{smaller
in the long run}.

Since $G$ is a Gr\"obner basis only once we complete the algorithm,
how can we compute the Hilbert function of its ideal? We do not! Instead,
we \emph{approximate} it using a \textbf{tentative Hilbert function}
$H_{R/\ideal{\lt\left(G\right)}}\left(d\right)$. This usually leads
us in the right direction, even when the polynomials are inhomogeneous~\cite{CaboaraDynAlg}.

\subsubsection{How do we choose the ordering?}

A \textbf{potential leading term} (PLT) of $r\in R$ is any term $t$
of $r$ for which there exists an admissible ordering $\sigma$ such
that $\lt[\sigma]\left(r\right)=t$. As long as $G$ is finite, there
is a finite set of equivalence classes of all monomial orderings.
We call each equivalence class a \textbf{cone associated to $G$},
and denote by $C\left(\sigma,G\right)$ the cone associated to $G$
that contains a specific ordering $\sigma$. When $g\in G$, we refer
to $C\left(\sigma,t,g\right)$ as \textbf{the cone associated to $G$
that guarantees} $\lt[\sigma]\left(g\right)=t$.

Suppose $t\in\supp\left(r\right)$ and we want to choose $\sigma$
such that $\lt[\sigma]\left(r\right)=t$. For each $u\in\supp\left(r\right)\backslash\left\{ t\right\} $,
we want $t>u$. This means $\sigma\mathbf{t}>\sigma\mathbf{u}$, or
$\sigma\left(\mathbf{t}-\mathbf{u}\right)>0$. A vector will suffice
to determine $\sigma$; and we can find such a vector using the system
of linear inequalities
\[
\lp\left(y,t,r\right)=\left\{ y_{k}>0\right\} _{k=1}^{n}\cup\left\{ \sum_{k=1}^{n}y_{k}\left(t_{k}-u_{k}\right)>0\right\} _{u\in\supp\left(r\right)\backslash\left\{ t\right\} }.
\]
(Here, $t_{k}$ and $u_{k}$ denote the $k$th entry of $\mathbf{t}$
and\textbf{ $\mathbf{u}$}, respectively. To avoid confusion with
the variables of the polynomial ring, we typically use $y$'s to denote
the unknown values of linear programs.) Since every cone $C\left(\sigma,G\right)$
is defined by some system of linear inequalities, this approach successfully
turns a difficult geometric problem into a well-studied algebraic
problem that we can solve using techniques from linear programming.
\begin{prop}[Propositions 1.5, 2.3 of \cite{CaboaraDynAlg}]
The cone $C\left(\sigma,F\right)$ can be described using a union
of such systems, one for each $f\in F$.\end{prop}
\begin{example}
In Cyclic-4, the choice of leading terms
\[
\left\{ x_{1},x_{1}x_{2},x_{1}x_{2}x_{3},x_{1}x_{2}x_{3}x_{4}\right\} 
\]
can be described by the system of linear inequalities
\begin{align*}
\left\{ y_{i}>0\right\} _{i=1}^{4} & \cup\left\{ y_{1}-y_{i}>0\right\} _{i=2}^{4}\\
 & \cup\left\{ y_{1}-y_{3}>0,y_{1}+y_{2}-y_{3}-y_{4}>0,y_{2}-y_{4}>0\right\} \\
 & \cup\left\{ y_{1}-y_{4}>0,y_{2}-y_{4}>0,y_{3}-y_{4}>0\right\} \\
 & \cup\left\{ y_{1}+y_{2}+y_{3}+y_{4}-0>0\right\} .
\end{align*}

\end{example}
The last inequality comes from the constraint $x_{1}x_{2}x_{3}x_{4}>1$,
and is useless. This illustrates an obvious optimization; if $u\mid t$
and $u\neq t$, the ordering's compatibility with division implies
that $u$ cannot be a leading term of $r$; the corresponding linear
inequality is trivial, and can be ignored. This leads to the first
of two criteria to eliminate terms that are not potential leading
terms.
\begin{prop}[The Divisibility Criterion; Proposition 2.5 and Corollary of \cite{CaboaraDynAlg}]
\label{prop: caboara first optimization}Let $r\in R$ and $t,u\in\supp\left(r\right)$.
If $u$ divides $t$ properly, then $t>_{\tau}u$ for every $\tau\in\allorders$.
In other words, $u$ is not a potential leading term of $r$.
\end{prop}
Caboara also proposed a second criterion based on the notion of \emph{refining}
the order. If $\sigma$ and $\tau$ are such that $C\left(\tau,G\right)\subseteq C\left(\sigma,G\right)$,
then we say that $\tau$ \textbf{refines} $\sigma$, or that $\tau$
\textbf{refines the order}. Caboara's implementation chooses an ordering
$\tau$ in line~\ref{enu: Select a new order} so that $\lt[\tau]\left(g\right)=\lt[\sigma]\left(g\right)$
for all $g\in G\backslash\left\{ r\right\} $, so that $\tau$ refines
$\sigma$. This allows the algorithm to discard line~\ref{enu: modify P after new order}
altogether, and this is probably a good idea in general.
\begin{prop}[The Refining Criterion; Proposition 2.6 of \cite{CaboaraDynAlg}]
\label{prop: caboara second optimization}Let $G=\left\{ g_{1},\ldots,g_{\ell}\right\} \subsetneq R$
and $t_{1},\ldots,t_{\ell}$ be potential leading terms of $g_{1},\ldots,g_{\ell}$,
respectively. The system
\[
\left\{ t_{k}>u:u\in\supp\left(g_{k}\right)\right\} _{k=1}^{\ell}
\]
is equivalent to the system
\[
\left\{ t_{k}>u:u\mbox{ a PLT of }g_{k}\mbox{ consistent w/}t_{j}=\lt\left(g_{j}\right)\;\forall j=1,\ldots,k-1\right\} _{k=1}^{\ell}.
\]

\end{prop}
Proposition~\ref{prop: caboara second optimization} implies that
we need only compare $t_{k}$ with other potential leading terms $u$
of $g_{k}$ that are consistent with the previous choices; we will
call such $u$, \textbf{compatible leading terms}.%
\footnote{This notion is essentially Caboara's notion of a potential leading
term with respect to $F$, $\lt\left(F\right)$. Our choice of different
vocabulary is allows us to emphasize that a ``potential'' leading
term for \emph{one} polynomial is not usually ``compatible'' with
previous choices.%
} From a practical point of view, refining the cone in this way is
a good idea, as the technique of refining the order allows one to
warm-start the simplex algorithm from a previous solution using the
dual simplex algorithm, lessening the overhead of linear programming.
It does require some record-keeping; namely, retaining and expanding
the linear program as we add new polynomials. This motivates the definition
of
\[
\lp\left(\sigma,G\right)=\lp\left(\sigma,\left\{ \left(\lt[\sigma]\left(g\right),g\right):g\in G\right\} \right):=\bigcup_{g\in G}\lp\left(y,\lt[\sigma]\left(g\right),g\right).
\]
This burden on space is hardly unreasonable, however, as these systems
would have to be computed even if we allow the order to change. That
approach would entail a combinatorial explosion.
\begin{example}
Let $F$ be the Cyclic-4 system. Suppose that, in the dynamic algorithm,
we add $\spoly\left(f_{1},0\right)=f_{1}$ to $G$, selecting $x_{1}$
for the leading term, with $\sigma=\left(2,1,1,1\right)$. Suppose
that next we select $\spoly\left(f_{2},0\right)=f_{2}$, reduce it
modulo $G$ to $r_{2}=x_{2}^{2}-2x_{2}x_{4}-x_{4}^{2}$, and select
$x_{2}^{2}$ as its leading term, with $\sigma=\left(3,2,1,1\right)$.
We have
\begin{align*}
\lp\left(\sigma,\left\{ \left(x_{1},f_{1}\right),\left(x_{2}^{2},f_{2}\right)\right\} \right) & =\phantom{\cup}\left\{ y_{k}>0\right\} _{k=1}^{4}\cup\left\{ y_{1}-y_{k}>0\right\} _{k=2}^{4}\\
 & \phantom{=\;}\cup\left\{ y_{2}-y_{4}>0,2y_{2}-2y_{4}>0\right\} .
\end{align*}
\end{example}
\begin{rem}
In a practical implementation, it is important to avoid redundant
constraints; otherwise, the programs quickly grow unwieldy. One way
to avoid redundant constraints is to put them into a canonical form
that allows us to avoid adding scalar multiples of known constraints.
Unfortunately, even this grows unwieldy before too long; one of this
paper's main points is to describe a method of minimizing the number
of required constraints.
\end{rem}

\subsubsection{When should we select the ordering?}

As noted in the introduction, Caboara's implementation refines the
ordering for a while, then permanently switches the refiner off. Making
this activation/deactivation mechanism more flexible is the major
goal of this investigation.

\subsubsection{Does the algorithm actually terminate?}

Termination is easy to see if you just refine orderings. In the more
general case, termination has been proved by Golubitsky~\cite{GolubitskyDynamicGB}.
He has found systems where it is advantageous to change the ordering,
rather than limit oneself to refinement.

\subsection{\label{sub: geometric view of dynamic algorithm}The geometric point
of view}

Here, we provide a visual interpretation of how the property $C\left(\tau,G\right)\subsetneq C\left(\sigma,G\right)$
affects the algorithm. This discussion was originally inspired by~\cite{GS93},
but we state it here in terms that are closer to the notion of a Gr\"obner
fan~\cite{MR88}.

Any feasible solution to the system of linear inequalities corresponds
to a half-line in the positive orthant. Thus, the set of all solutions
to any given system forms an open, convex set that resembles an infinite
cone. Adding polynomials to $G$ sometimes splits some of the cones.
(See Figure~\ref{fig: cones narrow as we add polys}). This gives
a geometric justification for describing Caboara's approach as a \textbf{narrowing
cone algorithm}.
\begin{figure}
\begin{centering}
\includegraphics[scale=0.2]{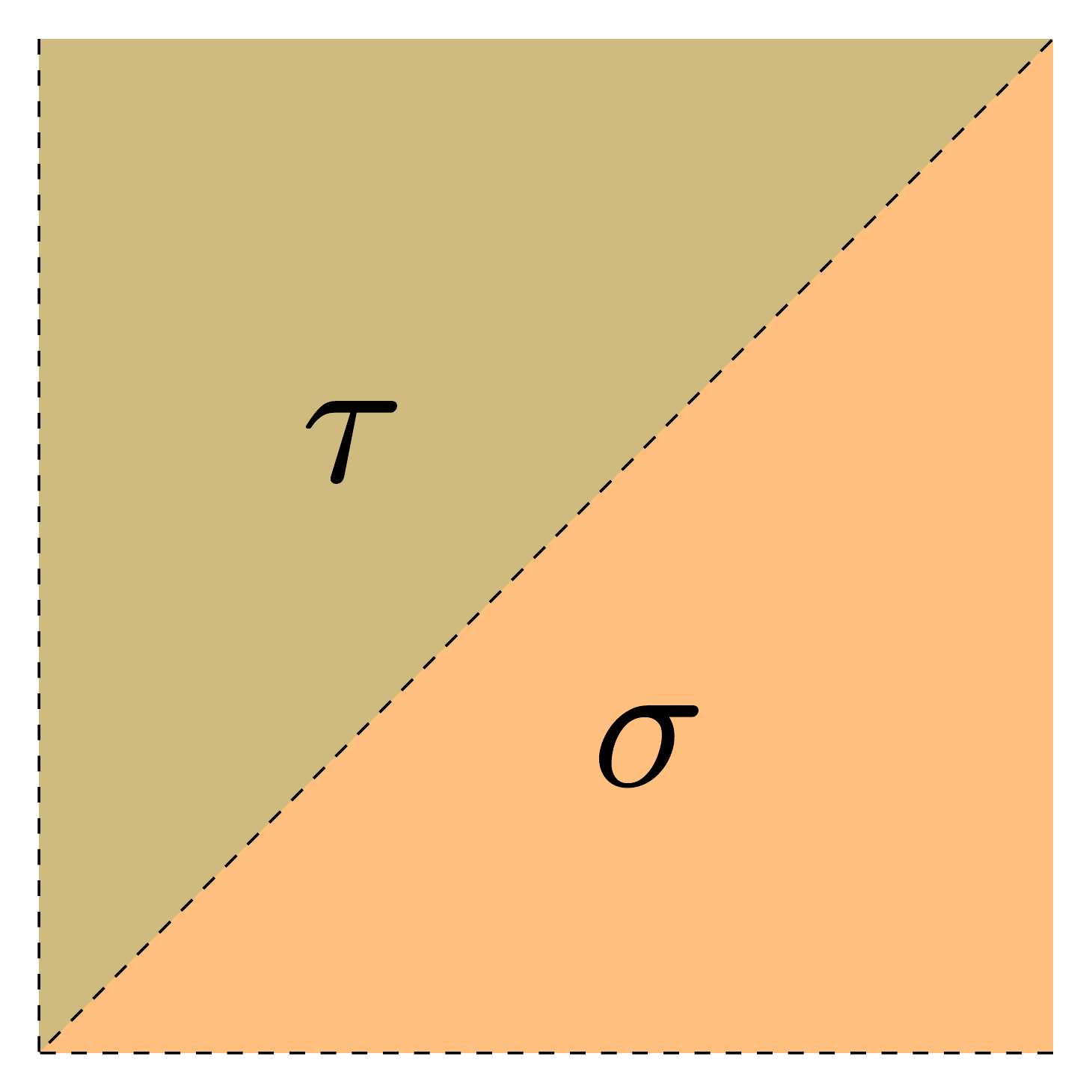}\quad\includegraphics[scale=0.2]{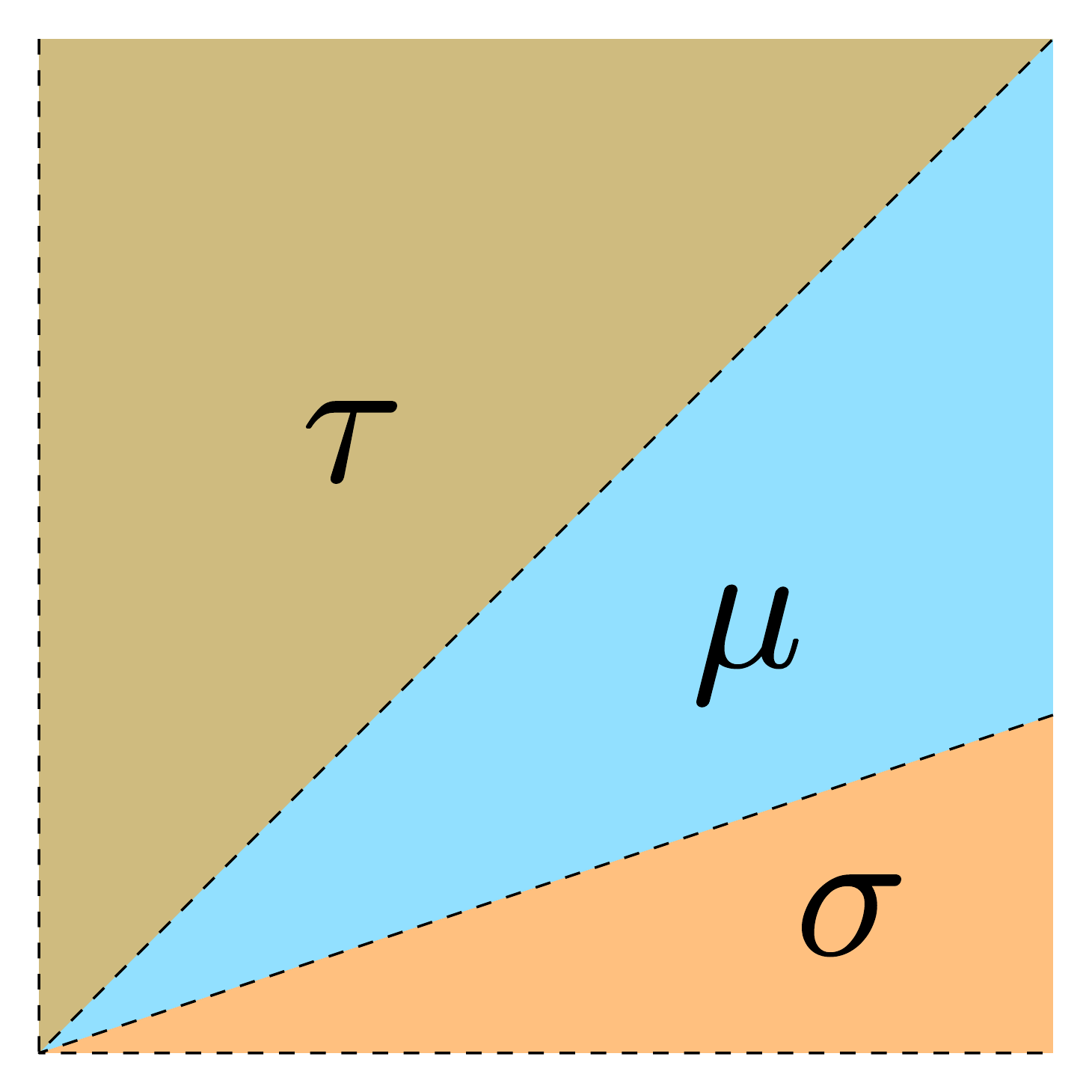}
\par\end{centering}

\caption{\label{fig: normal cones of newton polytope}\label{fig: cones narrow as we add polys}The
cones associated to a basis $G$ narrow as we add polynomials to the
basis. Here, $G=\left\{ x^{2}+y^{2}-4,xy-1\right\} $; the two possible
orderings are $\tau$ and $\sigma$, which give $\lt[\sigma]\left(G\right)=\left\{ x^{2},xy\right\} $
and $\lt[\tau]\left(G\right)=\left\{ y^{2},xy\right\} $. Suppose
we choose $\sigma$ and compute the $S$-poly\-nomial. This adds
$r=y^{3}+x-4y$ to the basis, and the cone containing $\sigma$ splits,
giving us two choices for $\lt\left(r\right)$. }
\end{figure}

Even though \emph{some} cones \emph{can} split when we add new polynomials,
not \emph{all} clones \emph{must} split. In particular, the cone containing
the desired ordering need not split, especially when the algorithm
is nearly complete. For this, reason, it is not necessary to refine
the cone every time a polynomial is added. The methods of the next
section help detect this situation.

\section{\label{sec: ideas}Exploiting the narrowing cone}

The main contribution of this paper is to use the narrowing cone to
switch the refiner on and off. We propose two techniques to accomplish
this: one keeps track of cones known to be disjoint (Section \ref{sub: disjoint cones});
the other keeps track of ``boundary vectors'' that prevent the refiner
from leaving the cone (Section \ref{sub: corner vectors}).

\subsection{\label{sub: disjoint cones}Disjoint Cones}

The Disjoint Cones Criterion is based on the simple premise that if
we track inconsistent constraints of linear programs, we can avoid
expanding the program later.

\subsubsection{Geometric motivation}

Let
\begin{itemize}
\item $C_{T}$ be the cone defined by the selection of $T=\left\{ t_{1},\ldots,t_{\ell}\right\} $
as the leading terms of $G=\left\{ g_{1},\ldots,g_{\ell}\right\} $,
\item $C_{u}$ be the cone defined by the selection of $u$ as the leading
term of $g_{\ell+1}$, and
\item $C_{T'}$ be the cone defined by the selection of $T'=\left\{ t_{1},\ldots,t_{\ell+k}\right\} $
as the leading terms of $G'=\left\{ g_{1},\ldots,g_{\ell+k}\right\} $.
\end{itemize}
Suppose the current ordering is $\sigma\in C_{T}$, and $C_{T}\cap C_{u}=\emptyset$.
It is impossible to refine the current ordering in a way that selects
$u$ as the leading term of $g_{\ell+1}$, as this would be inconsistent
with previous choices. On the other hand, if $\tau\in C_{T'}$ and
$C_{T'}\subseteq C_{T}$, then it is possible to refine $\sigma$
to an ordering $\tau$.

Now let
\begin{itemize}
\item $C_{v}$ be the cone defined by the selection of $v$ as the leading
term for $g_{\ell+k+1}$.
\end{itemize}
If $C_{v}\subseteq C_{u}$, then $C_{v}\cap C_{T'}\subseteq C_{u}\cap C_{T}=\emptyset$,
so we cannot refine $\sigma$ to any ordering that selects $v$ as
the leading term of $g_{\ell+k+1}$. Is there some way to ensure that
the algorithm does not waste time solving such systems of linear inequalities?
Yes! If we record the cone $C_{u}$, we can check whether $C_{v}\subseteq C_{u}$,
rather than going to the expense of building a linear program to check
whether $C_{v}\cap C_{T'}\neq\emptyset$.

\subsubsection{Algebraic implementation}

We could determine whether $C_{v}\subseteq C_{u}$ by building a linear
program, but we will content ourselves with determining when one set
of inconsistent linear constraints is a subset of another set.
\begin{thm}[Disjoint Cones Criterion]
Let $L_{T}$, $L_{U}$, and $L_{V}$ be sets of linear constraints.
If $L_{T}$ is inconsistent with $L_{U}$ and $L_{U}\subseteq L_{V}$,
then $L_{T}$ is also inconsistent with $L_{V}$.\end{thm}
\begin{proof}
Assume $L_{T}$ is inconsistent with $L_{U}$ and $L_{U}\subseteq L_{V}$.
The first hypothesis implies that $L_{T}\cap L_{U}=\emptyset$. The
second implies that $L_{V}$ has at least as many constraints as $L_{U}$,
so that the feasible regions $C_{U}$ and $C_{V}$, corresponding
to $L_{U}$ and $L_{V}$, respectively, satisfy the relation $C_{U}\supseteq C_{V}$.
Putting it all together, $C_{V}\cap C_{T}\subseteq C_{U}\cap C_{T}=\emptyset$.
\end{proof}
In our situation, $L_{U}$ and $L_{V}$ correspond to different choices
of leading terms of a new polynomial added to the set, while $L_{T}=\lp\left(\sigma,G\right)$.
We want to consider both the case where $L_{V}$ is a set of \emph{new}
constraints, and the case where $L_{V}$ is some extension of $L_{T}$.
Rather than discard the inconsistent linear program $L_{U}$, we will
retain it and test it against subsequent sets of constraints, avoiding
pointless invocations of the simplex algorithm.

We implement the geometric idea using a global variable, $\rejects$.
This is a set of sets; whenever $L_{T}$ is known to be consistent,
but the simplex algorithm finds $L_{T}\cup L_{U}$ inconsistent, we
add $L_{U}$ to\emph{ $\mathit{rejects}$}. Subsequently, while creating
the constraints in an extension $L_{V}$ of $L_{T}$, we check whether
$L_{U}$ is contained in either $L_{V}$ or $L_{T}\cup L_{V}$; if
so, we reject $L_{V}$ out of hand, without incurring the burden of
the simplex algorithm.

Again, we are not checking whether the cones are disjoint, only the
necessary condition of whether the linear constraints are a subset.
With appropriate data structures, the complexity of determining subset
membership is relatively small; with hashed sets, for example, the
worst-case time complexity would be the cost of the hash function
plus $O\left(\left|\mathit{rejects}\right|\right)$.

\subsection{\label{sub: corner vectors}Boundary Vectors}

Unlike the Disjoint Cones Criterion, the Boundary Vectors Criterion
can prevent the construction of \emph{any }constraints.

\subsubsection{Geometric motivation}

Let $C\left(G,\sigma\right)$ be a cone associated with $G$, containing
the ordering $\sigma$.
\begin{defn}
The \textbf{closure} of $C\left(G,\sigma\right)$ is the feasible
region obtained by rewriting $\lp\left(\sigma,G\right)$ as inclusive
inequalities ($\geq$ in place of $>$). Let $d\in\mathbb{R}$ be
positive. We say that $\omega\in\mathbb{R}^{n}$ is a \textbf{boundary
vector} of $C\left(G,\sigma\right)$ if it is an extreme point of
the intersection of the closure of $C\left(G,\sigma\right)$ and the
additional constraint $\sum_{k=1}^{n}x_{k}=d$.

For a fixed $d$, we denote the set of all boundary vectors of $C\left(G,\sigma\right)$
by $\Omega_{G,\sigma,d}$. When the value of $d$ is not critical
to the discussion, we simply write $\Omega_{G,\sigma}$. Likewise,
if the values of $G$ and $\sigma$ are not critical to the discussion,
or if they are understood from context, we simply write $\Omega$.\end{defn}
\begin{example}
\label{exa: corner vectors}Suppose $\sigma=\left(18,5,7\right)$
and $C\left(G,\sigma\right)$ is defined by
\[
\left\{ \begin{array}{rl}
2y_{1}-y_{2} & >0\\
-y_{1}+4y_{2} & >0\\
y_{1}+y_{2}-3y_{3} & >0\\
-y_{2}+y_{3} & >0
\end{array}\right..
\]
When $d=30$, we have $\Omega=\left\{ \left(15,7.5,7.5\right),\left(20,5,5\right),\left(18,4.5,7.5\right)\right\} $.
Not all intersections of constraints are boundary vectors; one intersection
that does not border the feasible region is $\left(22.5,0,7.5\right)$,
which satisfies all but the third constraint.
\end{example}
Boundary vectors possess the very desirable property of capturing
all possible refinements of the term ordering.
\begin{thm}[Boundary Vectors Criterion]
\label{thm: can consider only boundary vectors}Let $r\in R$ and
$\Omega$ the set of boundary vectors of $C\left(\sigma,G\right)$.
Write $t=\lt[\sigma]\left(r\right)$. If there exists $\tau\in C\left(\sigma,G\right)$
such that $\lt[\tau]\left(r\right)=u\neq t$ --- that is, there exists
$\tau$ that refines the order differently from $\sigma$ --- then
there exists $\omega\in\Omega$ such that $\omega\left(\mathbf{u}-\mathbf{t}\right)>\mathbf{0}$.
\end{thm}
Two observations are in order before we prove Theorem~\ref{thm: can consider only boundary vectors}.
First, the converse of Theorem~\ref{thm: can consider only boundary vectors}
is not true in general. For example, let $\omega=\left(2,1\right)$,
$r=x^{2}+x+y$, $v=x^{2}$, $u=x$, and $t=y$. Even though $\omega\left(\mathbf{u}-\mathbf{t}\right)>0$,
the fact that $u\mid v$ implies that no admissible ordering $\tau$
chooses $\lt[\tau]\left(r\right)=u$.

Nevertheless, the theorem does imply a useful corollary:
\begin{cor}
If we know the set $\Omega$ of boundary vectors for $C\left(\sigma,G\right)$,
then we can discard any term $u$ such that $\omega\left(\mathbf{t}-\mathbf{u}\right)>0$
for all $\omega\in\Omega$. That is, $u$ is not a compatible leading
term.
\end{cor}
We turn now to the proof of Theorem~\ref{thm: can consider only boundary vectors}.
Figure~\ref{fig: cross-section of normal cone} illustrates the intuition;
\begin{figure}
\begin{centering}
\includegraphics[scale=0.4]{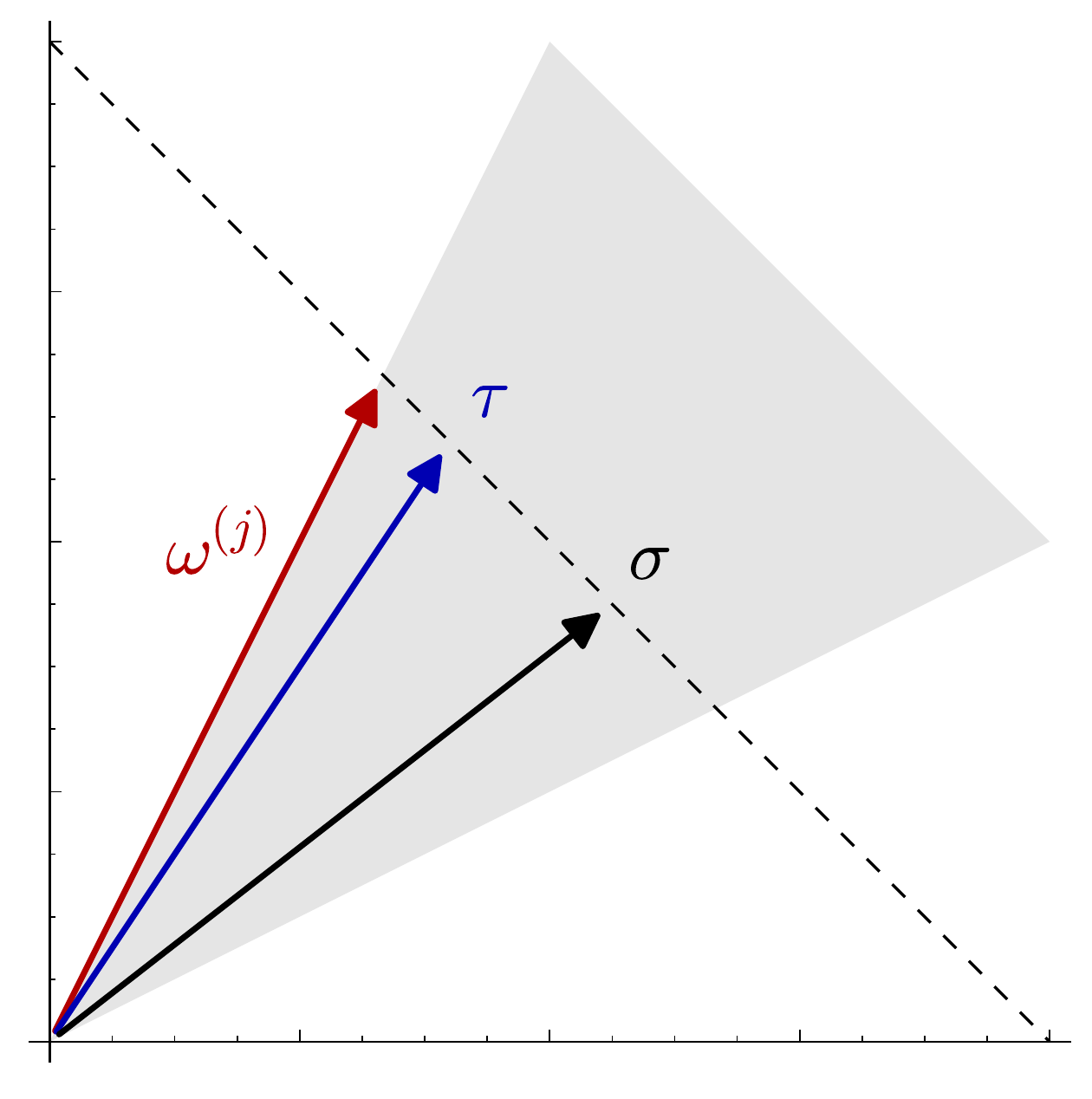}
\par\end{centering}

\caption{\label{fig: cross-section of normal cone}If $\lt[\tau]\left(r\right)\neq\lt[\sigma]\left(r\right)$,
convexity and linearity imply that we can find a boundary vector $\omega$
such that $\omega$ would give $\lt[\tau]\left(r\right)$ more weight
than $\lt[\sigma]\left(r\right)$.}
\end{figure}
 here, $\tau$ and $\sigma$ select different leading terms, and $\omega$
lies on the other side of $\tau$ from $\sigma$. By linearity, $\omega$
gives greater weight to $\lt[\tau]\left(r\right)$ than to $\lt[\sigma]\left(r\right)$.
\begin{proof}
[of Theorem \ref{thm: can consider only boundary vectors}]Suppose
that there exists $\tau\in C\left(\sigma,G\right)$ such that $\lt[\tau]\left(r\right)=u$.
By definition, $\tau\left(\mathbf{u}-\mathbf{t}\right)>0$. Let $d=\sum\tau_{k}$;
if $\tau=\omega$ for some $\omega\in\Omega_{G,\sigma,d}$, then we
are done. Otherwise, consider the linear program defined by maximizing
the objective function $\sum y_{k}\left(u_{k}-t_{k}\right)$ subject
to the closure of $\lp\left(\sigma,G\right)\cup\left\{ \sum y_{k}=\sum\tau_{k}\right\} $.
This is a convex set; the well-known Corner Point Theorem implies
that a maximum of any objective function occurs at an extreme point~\cite{CalvertLinearProgramming}\cite{SultanLinearProgramming}.
By definition, such a point is a boundary vector of $C\left(\sigma,G\right)$.
Let $\omega\in C\left(\sigma,G\right)$ be a boundary vector where
$\sum y_{k}\left(u_{k}-t_{k}\right)$ takes its maximum value; then
$\omega\left(\mathbf{u}-\mathbf{t}\right)\geq\tau\left(\mathbf{u}-\mathbf{t}\right)>0$.
\end{proof}

Computing $\Omega$ can be impractical, as it is potentially exponential
in size. We approximate it instead by computing corner points that
correspond to the maximum and minimum of each variable on a cross
section of the cone with a hyperplane, giving us at most $2n$ points.
Figure~\ref{fig: underestimating Omega}
\begin{figure}
\begin{centering}
\includegraphics[scale=0.13]{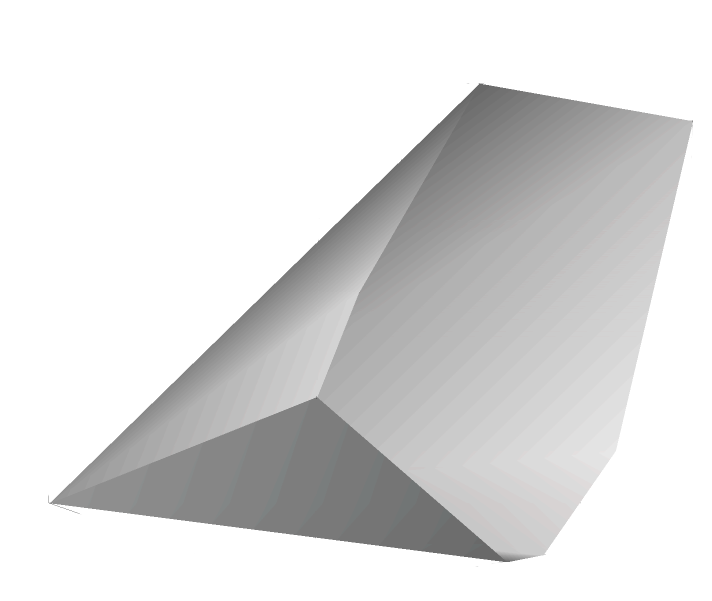}\qquad\includegraphics[scale=0.17]{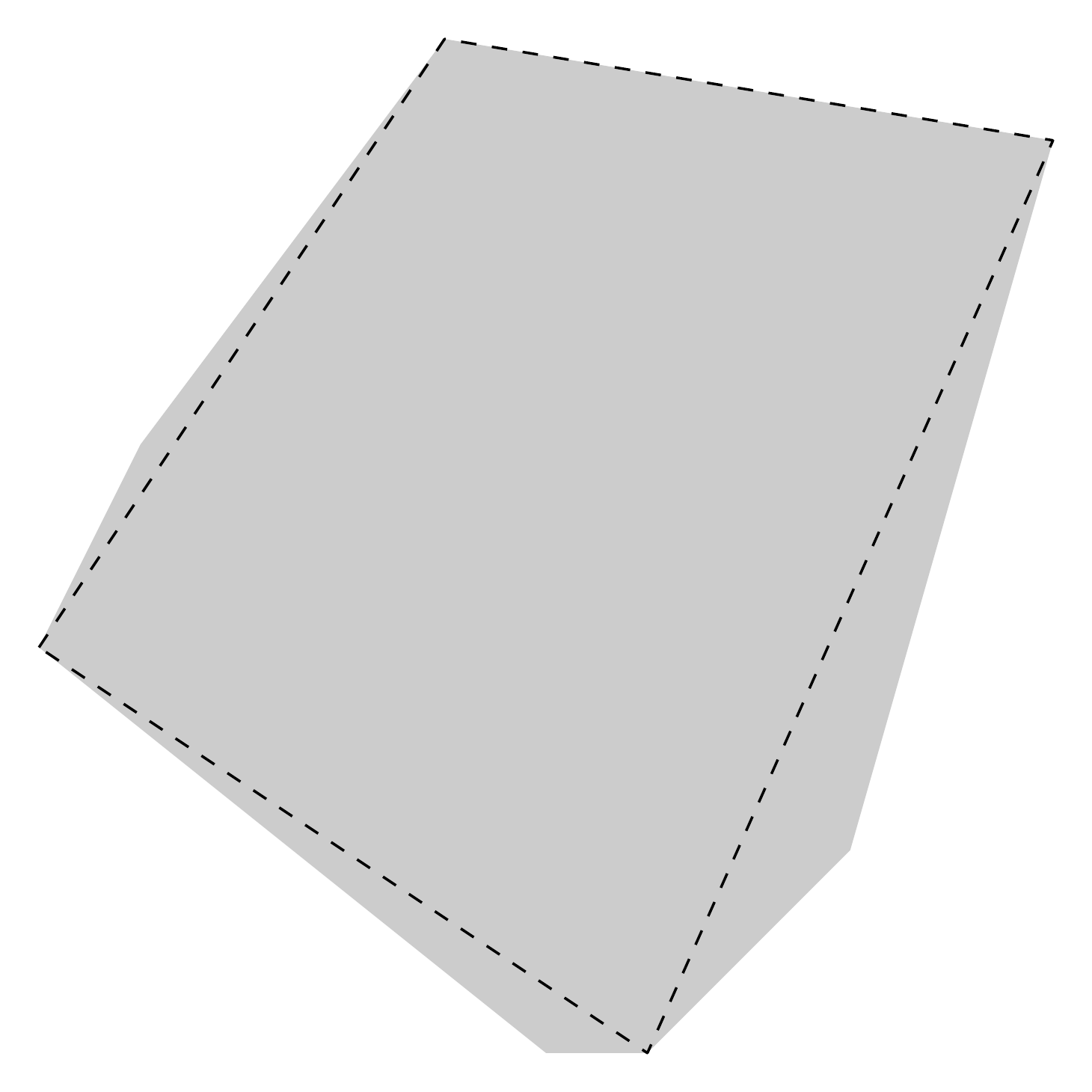}
\par\end{centering}

\caption{\label{fig: underestimating Omega}We approximate $\Omega$ by computing
boundary vectors corresponding to points that maximize and minimize
the value of an objective function. The cone at the left has seven
boundary vectors, as we see in the cross section on the right. In
this case, four vectors maximize and minimize the variables; they
define a ``sub-cone'' whose cross-section corresponds to the dashed
line.}
\end{figure}
 illustrates the idea.
\begin{example}
\label{exa: corner vectors obtained by maximizing, minimizing variables}Continuing
Example~\ref{exa: corner vectors}, maximizing $y_{1}$, $y_{2}$,
and $y_{3}$ gives us the boundary vectors listed in that example.
In this case, these are \emph{all} the boundary vectors, but we are
not always so lucky.
\end{example}
Approximating $\Omega$ has its own disadvantage; inasmuch as some
orderings are excluded, we risk missing some refinements that could
produce systems that we want. We will see that this is not a serious
drawback in practice.

\subsubsection{\label{sub: avoiding changes of lm}Minimizing the number of constraints}

While boundary vectors reduce the \emph{number} of linear programs
computed, the \emph{size} of the linear programs can remain formidable.
After all, we are still adding constraints for every monomial that
passes the Divisibility Criterion. Is there some way to use boundary
vectors to minimize the number of constraints in the program?

We will attempt to add only those constraints that correspond to terms
that the boundary vectors identify as compatible leading terms. As
we are not computing all the boundary vectors, the alert reader may
wonder whether this is safe.
\begin{example}
Suppose
\begin{itemize}
\item $\mu\in C\left(\tau,\left\{ g_{1},\ldots,g_{\ell+k}\right\} \right)\subsetneq C\left(\sigma,\left\{ g_{1},\ldots,g_{\ell}\right\} \right)$,
\item $t=\lt[\sigma]\left(g_{\ell}\right)$, and
\item $u=\lt[\mu]\left(g_{\ell}\right)$.
\end{itemize}
Suppose further that, when $g_{\ell}$ is added to the basis, the
algorithm selects $\sigma$ for the ordering.

For some choices of boundary vectors, the algorithm might not notice
that $\mu\in C\left(\sigma,\left\{ g_{1},\ldots,g_{\ell}\right\} \right)$,
as in Figure~\ref{fig: why we sometimes have to walk the algorithm back}(a).
\begin{figure}
\begin{centering}
\begin{tabular}{ccc}
\includegraphics[scale=0.2]{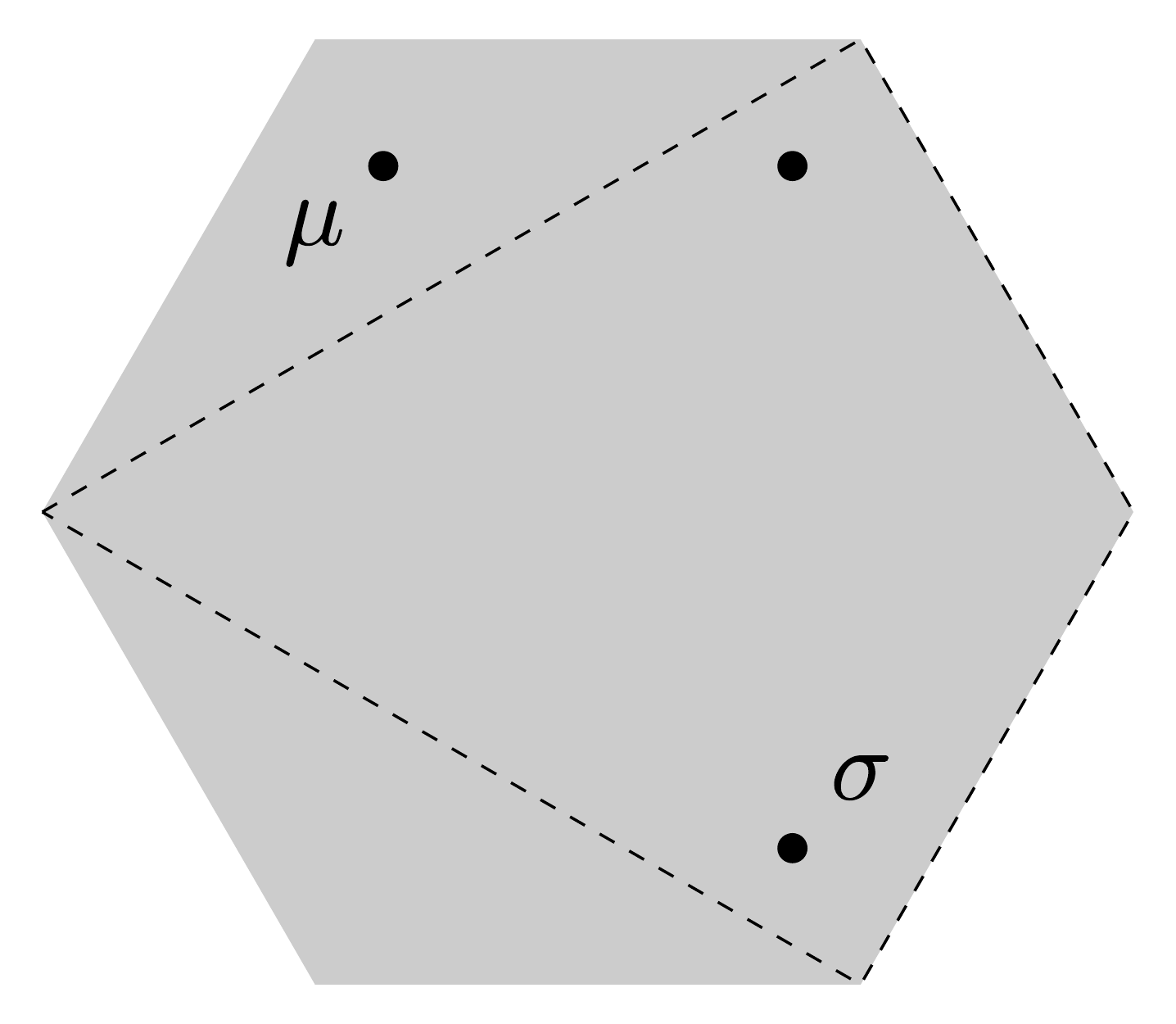} &  & \includegraphics[scale=0.2]{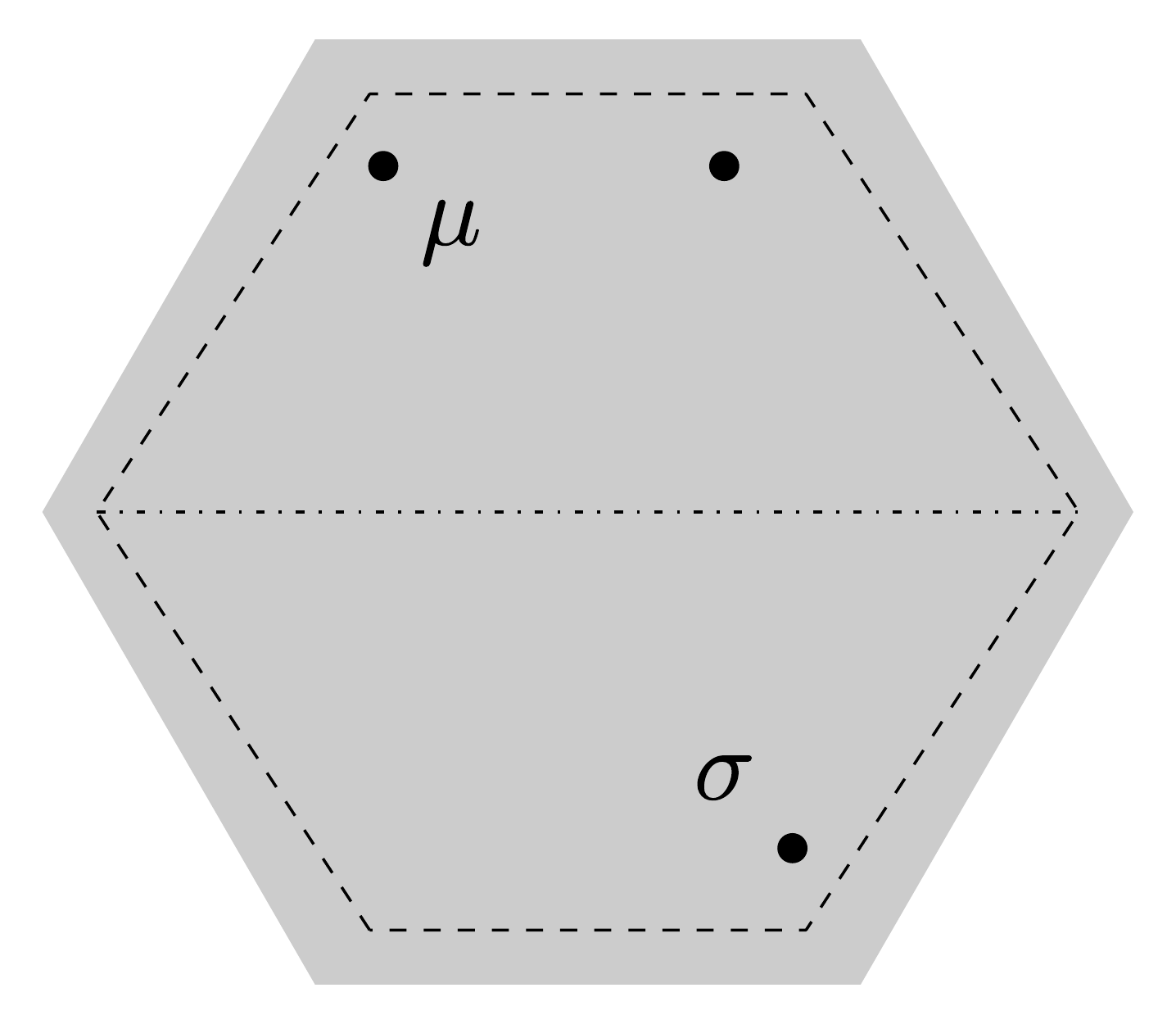}\tabularnewline
(a) &  & (b)\tabularnewline
\end{tabular}
\par\end{centering}

\caption{\label{fig: why we sometimes have to walk the algorithm back}Although
$\sigma$ and $\mu$ lie within the same cone, the choice of border
vectors in (a) could mean the algorithm at first does not add a constraint
to guarantee $\lt[\sigma]\left(g_{\ell}\right)>\lt[\mu]\left(g_{\ell}\right)$.
If the narrowed cone splits later on, as in (b), and the algorithm
moves into a subcone that does not contain $\sigma$, a subsequent
choice of $\mu$ is possible, causing a change in the leading terms.
We can try adding the previously-overlooked constraint, and continue
if we find a feasible solution. The unlabeled dot could represent
such a compromise ordering.}
\end{figure}
 Thus, it would not add the constraint $\left(y_{1},\ldots,y_{n}\right)\cdot\left(\mathbf{t}-\mathbf{u}\right)>0$.
A later choice of boundary vectors \emph{does} recognize that $\mu\in C\left(\tau,\left\{ g_{1},\ldots,g_{\ell+k}\right\} \right)$,
and even selects $\mu$ as the ordering. We can see this in Figure
\ref{fig: why we sometimes have to walk the algorithm back}(b). In
this case, the leading term of $g_{\ell}$ changes from $t$ to $u$;
the ordering has been changed, not refined!
\end{example}
Since the Gr\"obner basis property depends on the value of the leading
terms, this endangers the algorithm's correctness. Fortunately, it
is easy to detect this situation; the algorithm simply monitors the
leading terms.

It is not so easy to remedy the situation once we detect it. One approach
is to add the relevant critical pairs to $P$, and erase any record
of critical pairs computed with $g_{\ell}$, or discarded because
of it. Any practical implementation of the dynamic algorithm would
use criteria to discard useless critical pairs, such as those of Buchberger,
but if the polynomials' leading terms have changed, the criteria no
longer apply. Recovering those pairs would require the addition of
needless overhead.

A second approach avoids these quandaries: simply add the missing
constraint to the system. This allows us to determine whether we were
simply unlucky enough to choose $\sigma$ from a region of $C\left(\sigma,G\right)$
that lies outside $C\left(\mu,G\right)$, or whether really there
is no way to choose both $\lt[\sigma]\left(g_{\ell}\right)$ and $\lt[\mu]\left(g_{\ell+k}\right)$
simultaneously. In the former case, the linear program will become
infeasible, and we reject the choice of $\mu$; in the latter, we
will be able to find $\tau\in C\left(\sigma,G\right)\cap C\left(\mu,G\right)$.

\subsubsection{Implementation}

Since the constraints of $\lp\left(\sigma,G\right)$ consist of integer
polynomials, it is possible to find integer solutions for the boundary
vectors; one simply rescales rational solutions once they are found.
However, working with exact arithmetic can be quite slow, techniques
of integer programming are \emph{very} slow, and for our purposes,
floating-point approximations are quite suitable. Besides, most linear
solvers work with floating point numbers.

On the other hand, using floating point introduces a problem when
comparing terms. The computer will sometimes infer $\omega\cdot\left(\mathbf{u}-\mathbf{t}\right)>0$
even though the exact representation would have $\omega\cdot\left(\mathbf{u}-\mathbf{t}\right)=0$.
We can get around this by modifying the constraints of the linear
program to $\omega\cdot\left(\mathbf{u}-\mathbf{t}\right)\geq\epsilon$
for some sufficiently large $\epsilon>0$. As we see in Figure~\ref{fig: underestimating Gamma},
\begin{figure}
\begin{centering}
\includegraphics[scale=0.17]{underestimating_corners}\qquad\includegraphics[scale=0.17]{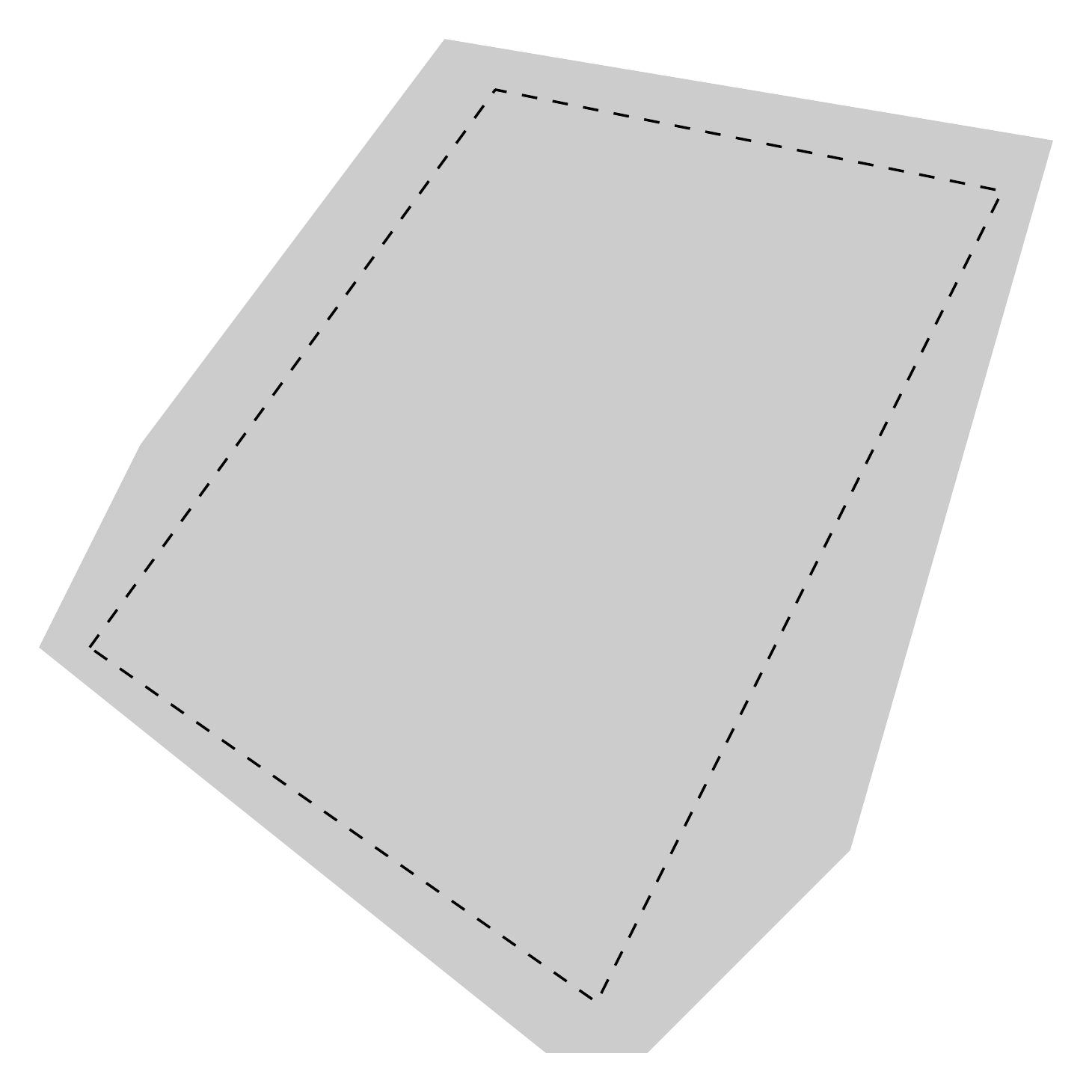}
\par\end{centering}

\caption{\label{fig: underestimating Gamma}Estimation of the corner points
of $C\left(\sigma,G\right)$ by a perturbation of the inequalities.
Each diagram shows a two-dimensional cross-section of a three-dimensional
cone. The dashed lines on the connect actual boundary vectors. The
diagram on the right connects boundary vectors of a perturbation of
the system of linear inequalities. The perturbation is designed to
give us \emph{interior} points of the cone that are close to corner
points.}
\end{figure}
 the polygon no longer connects extrema, but points that approximate
them. We might actually reject some potential leading terms $u$ on
this account, but on the other hand, we \emph{never} waste time with
terms that are incompatible with the current ordering.

This modified linear program is in fact useful for computing feasible
points to the original linear program as well.
\begin{thm}
\label{thm: modified LP has sol iff original LP has sol}Let $\epsilon>0$
and $J$ a finite subset of $\mathbb{N}$. The system of linear inequalities
\[
\lp\left(\sigma,G\right)=\left\{ \mathbf{a}^{\left(j\right)}\cdot\left(y_{1},\ldots,y_{n}\right)>0\right\} _{j\in J}\cup\left\{ y_{j}>0\right\} _{j=1}^{n}
\]
is feasible if and only if the linear program
\[
\left\{ \mathbf{a}^{\left(j\right)}\cdot\left(y_{1},\ldots,y_{n}\right)\geq\epsilon\right\} _{j\in J}\cup\left\{ y_{j}\geq\epsilon\right\} _{j=1}^{n}
\]
is also feasible.\end{thm}
\begin{proof}
A solution to the linear program obviously solves the system of linear
inequalities. Thus, suppose the system of linear inequalities has
a solution $\sigma$. Let $\tau=b\sigma$, where $b\in\mathbb{N}$
is chosen large enough that $\tau_{k}\geq\epsilon$ for $k=1,\ldots,n$.
For each $j\in J$, define
\[
\gamma_{j}=\sum_{k=1}^{n}a_{k}^{\left(j\right)}\tau_{k},
\]
then choose $c_{j}\geq1$ such that $c_{j}\gamma_{j}\geq\epsilon$.
Put $d=\max\left\{ c_{j}\right\} _{j\in J}$. Let $\omega=d\tau$.
We have $\omega_{k}\geq\tau_{k}\geq\epsilon$ for each $k=1,\ldots,n$,
and
\[
\mathbf{a}^{\left(j\right)}\cdot\omega=\sum_{k=1}^{n}a_{k}^{\left(j\right)}\omega_{k}=d\sum_{k=1}^{n}a_{k}^{\left(j\right)}\tau_{k}=d\gamma_{j}\geq c_{j}\gamma_{j}\geq\epsilon
\]
for each $j\in J$. We have shown that $\omega$ is a solution to
the linear program, which means that the linear program is also feasible.
\end{proof}
Based on Theorem~\ref{thm: modified LP has sol iff original LP has sol},
we take the following approach:
\begin{enumerate}
\item Replace each constraint $\mathbf{a}^{\left(j\right)}\cdot y>0$ of
$\lp\left(\sigma,G\right)$ with $\mathbf{a}^{\left(j\right)}\cdot y\geq\epsilon$,
add constraints $y_{k}\geq\epsilon$ for $k=1,\ldots,n$, and take
as the objective function the \emph{minimization} of\emph{ $\sum y_{k}$}.
We denote this new linear program as $\mlp\left(\sigma,G\right)$.
\item Solve $\mlp\left(\sigma,G\right)$. This gives us a vector $\tau$
that can serve as a weighted ordering for the terms already computed.
\item Identify some $d\in\mathbb{R}$ such that $\mlp\left(\sigma,G\right)$
intersects the hyperplane $\sum y_{k}=d$, giving us a cross-section
$K$ of the feasible region. This is trivial once we have a solution
$\tau$ to $\mlp\left(\sigma,G\right)$, since we can put $d=1+\sum\tau_{k}$.
\item Compute an approximation to $\Omega$ by maximizing and minimizing
each $y_{k}$ on $K$.
\end{enumerate}
Algorithm \emph{compute\_boundary\_vectors (}Figure~\ref{alg: computing hypercube})\textbf{}
\begin{figure}
\rule[0.5ex]{1\columnwidth}{1pt}

\begin{raggedright}
\textbf{algorithm} \emph{compute\_boundary\_vectors}
\par\end{raggedright}

\begin{raggedright}
\textbf{inputs:}
\par\end{raggedright}
\begin{itemize}
\item $L=\mlp\left(\sigma,G\right)$
\item $\tau\in\mathbb{R}^{n}$ that solves $L$
\end{itemize}
\begin{raggedright}
\textbf{outputs:} $\Psi\subsetneq\mathbb{R}^{n}$, a set of vectors
that approximates the boundary vectors of $\mlp\left(\sigma,G\right)$
\par\end{raggedright}

\begin{raggedright}
\textbf{do:}
\par\end{raggedright}
\begin{enumerate}
\item let $\Psi=\left\{ \right\} $
\item let $d=1+\tau_{1}+\cdots+\tau_{n}$
\item let $L=L\cup\left\{ y_{1}+\cdots+y_{n}=d\right\} $
\item \textbf{for} $k\in\left\{ 1,\ldots,n\right\} $

\begin{enumerate}
\item add to $\Psi$ the solution $\omega$ of $L$ that maximizes $\omega_{k}$
\item add to $\Psi$ the solution $\omega$ of $L$ that minimizes $\omega_{k}$
\end{enumerate}
\item \label{enu: return hyperrect}\textbf{return} $\Psi$
\end{enumerate}
\caption{\label{alg: computing hypercube}Algorithm to compute approximate
boundary vectors to $C\left(G,\sigma\right)$}

\rule[0.5ex]{1\columnwidth}{1pt}
\end{figure}
 gives pseudocode to do this; it generates a set of boundary vectors
$\Psi$ that approximates the set $\Omega$ of boundary vectors of
$C\left(\sigma,G\right)$.

Once we have an approximation $\Psi$ to the boundary vectors $\Omega$,
we use it to eliminate terms that cannot serve as leading terms within
the current cone. Algorithm \emph{identify\_clts\_using\_boundary\_vectors}
(Figure~\ref{alg: using corner approximations}),\textbf{}
\begin{figure}
\rule[0.5ex]{1\columnwidth}{1pt}

\begin{raggedright}
\textbf{algorithm} \emph{identify\_clts\_using\_boundary\_vectors}
\par\end{raggedright}

\begin{raggedright}
\textbf{inputs:}
\par\end{raggedright}
\begin{itemize}
\item $\sigma\in\allorders$, the current term ordering
\item $t=\lt[\sigma]\left(r\right)$, where $r\in R$
\item $U=\supp\left(r\right)\backslash\left\{ t\right\} $
\item $\Psi\subsetneq\mathbb{R}^{n}$, approximations to the boundary vectors
of $C\left(G,\sigma\right)$
\end{itemize}
\begin{raggedright}
\textbf{outputs:} $V$, where $v\in V$ iff $v=t$, or $v\in U$ and
$\psi\left(\mathbf{v}-\mathbf{t}\right)>0$ for some $\psi\in\Psi$
\par\end{raggedright}

\begin{raggedright}
\textbf{do:}
\par\end{raggedright}
\begin{enumerate}
\item let $V=\left\{ t\right\} $
\item \textbf{for} $u\in U$

\begin{enumerate}
\item \textbf{\label{enu: is there psi such that u is larger than t?}if}
$\psi\left(\mathbf{u}-\mathbf{t}\right)>0$ for some $\psi\in\Psi$

\begin{enumerate}
\item add $u$ to $V$
\end{enumerate}
\end{enumerate}
\item return $V$
\end{enumerate}
\caption{\label{alg: using corner approximations}Eliminating terms using approximation
to boundary vectors}

\rule[0.5ex]{1\columnwidth}{1pt}
\end{figure}
 accomplishes this by looking for $u\in\supp\left(r\right)\backslash\left\{ \lt[\sigma]\left(r\right)\right\} $
and $\psi\in\Psi$ such that $\psi\left(\mathbf{u}-\mathbf{t}\right)>0$.
If it finds one, then $u$ is returned as a compatible leading term.

In Section~\ref{sub: avoiding changes of lm}, we pointed out that
creating constraints only for the terms identified as potential leading
terms by the use of boundary vectors can lead to an inconsistency
with previously-chosen terms. For this reason, we not only try to
solve the linear program, but invokes algorithm \emph{monitor\_lts}
(Figure~\ref{alg: ensuring monomials invariant}) to verify that
previously-determined leading terms remain invariant.
\begin{figure}
\rule[0.5ex]{1\columnwidth}{1pt}

\begin{raggedright}
\textbf{algorithm} \emph{monitor\_lts}
\par\end{raggedright}

\begin{raggedright}
\textbf{inputs}
\par\end{raggedright}
\begin{itemize}
\item $G$, the working basis
\item $\sigma\in\allorders$, the old ordering
\item $\tau\in\allorders$, a new ordering
\item $L=\mlp\left(\tau,G\right)$
\end{itemize}
\begin{raggedright}
\textbf{outputs}
\par\end{raggedright}
\begin{itemize}
\item (\texttt{True}, $\mu$) if there exists $\mu\in L$ refining $C\left(G,\sigma\right)$
and $\lt[\mu]\left(g_{\textrm{last}}\right)=\lt[\tau]\left(g_{\textrm{last}}\right)$,
where $g_{\textrm{last}}$ is the newest element of $G$
\item \texttt{False} otherwise
\end{itemize}
\begin{raggedright}
\textbf{do:}
\par\end{raggedright}
\begin{enumerate}
\item let $\mu=\tau$
\item \textbf{while} there exists $g\in G$ such that $\lt[\mu]\left(g\right)\neq\lt[\sigma]\left(g\right)$

\begin{enumerate}
\item \textbf{for each} $g\in G$ whose leading term changes

\begin{enumerate}
\item let\textbf{ }$t=\lt[\sigma]\left(g\right)$, $u=\lt[\mu]\left(g\right)$
\item let $L=L\bigcup\left\{ y\cdot\left(\mathbf{t}-\mathbf{u}\right)\right\} $
\end{enumerate}
\item \textbf{if} $L$ is infeasible \textbf{return} \texttt{False}\textbf{}\\
\textbf{else} let $\mu$ be the solution to $L$
\end{enumerate}
\item \textbf{return} (\texttt{True}, $\mu$)
\end{enumerate}
\caption{\label{alg: ensuring monomials invariant}Ensuring the terms remain
invariant}

\rule[0.5ex]{1\columnwidth}{1pt}
\end{figure}
 If some leading terms would change, the algorithm obtains a compromise
ordering whenever one exists.
\begin{thm}
\label{thm: monitor_lts correct}Algorithm \texttt{monitor\_lts} of
Figure~\ref{alg: ensuring monomials invariant} terminates correctly.\end{thm}
\begin{proof}
Termination is evident from the fact that $G$ is a finite list of
polynomials, so the while loop can add only finitely many constraints
to $L$. Correctness follows from the fact that the algorithm adds
constraints to $\mlp\left(\tau,G\right)$ if and only if they correct
changes of the leading term. Thus, it returns (\texttt{True}, $\mu$)
if and only if it is possible to build a linear program $L$ whose
solution $\mu$ lies in the non-empty set $C\left(\tau,G\right)\cap C\left(\sigma,G\right)$.
\end{proof}
It remains to put the pieces together. 
\begin{figure}
\rule[0.5ex]{1\columnwidth}{1pt}

\begin{raggedright}
\textbf{algorithm} \emph{dynamic\_algorithm\_with\_geometric\_criteria}
\par\end{raggedright}

\begin{raggedright}
\textbf{inputs:} $F\subseteq R$
\par\end{raggedright}

\begin{raggedright}
\textbf{outputs:} $G\subseteq R$ and $\sigma\in\allorders$ such
that $G$ is a Gr\"obner basis of $\ideal{F}$ with respect to $\sigma$
\par\end{raggedright}

\begin{raggedright}
\textbf{do:}
\par\end{raggedright}
\begin{enumerate}
\item Let $G=\left\{ \right\} $, $P=\left\{ \left(f,0\right):f\in F\right\} $,
$\sigma\in\allorders$
\item Let $\textit{rejects}=\left\{ \right\} $, $\Psi=\left\{ e_{k}:k=1,\ldots,n\right\} $
\item \textbf{while} $P\neq\emptyset$\label{enu: while loop of dynamic algorithm-1}

\begin{enumerate}
\item Select $\left(p,q\right)\in P$ and remove it
\item \label{enu: r full reduction of s(p,q) (dynamic)-1}Let $r$ be a
remainder of $\spoly\left(p,q\right)$ modulo $G$
\item \textbf{if} $r\neq0$

\begin{enumerate}
\item Add $\left(g,r\right)$ to $P$ for each $g\in G$
\item Add $r$ to $G$
\item \label{enu: Select a new order-1}Select $\sigma\in\allorders$, using
\emph{identify\_clts\_using\_boundary\_vectors} to eliminate incompatible
terms and \emph{monitor\_lts} to ensure consistency of $\tau$, storing
failed linear programs in $\textit{rejects}$
\item Remove useless pairs from $P$
\end{enumerate}
\item Let $\Psi=$\emph{ compute\_boundary\_vectors}($L$, $\mlp\left(\sigma,G\right)$)
\end{enumerate}
\item \textbf{return} $G$, $\sigma$
\end{enumerate}
\caption{\label{alg: Dynamic Algorithm with New Criteria}A dynamic Buchberger
algorithm that employs the Disjoint Cones and Boundary Vectors criteria}

\rule[0.5ex]{1\columnwidth}{1pt}
\end{figure}

\begin{thm}
Algorithm \texttt{\textup{\emph{dynamic\_algorithm\_with\_geometric\_criteria}}}\textup{\emph{
terminates correctly.}}\end{thm}
\begin{proof}
The only substantive difference between this algorithm and the dynamic
algorithm presented in~\cite{CaboaraDynAlg} lies in line \ref{enu: Select a new order-1}.
In the original, it reads\end{proof}
\begin{quote}
$\sigma:=\mathbf{RefineCurrentOrder}\left(f_{t+1},F,\sigma\right).$
\end{quote}
Theorems~\ref{thm: can consider only boundary vectors},~\ref{thm: modified LP has sol iff original LP has sol},
and~\ref{thm: monitor_lts correct} are critical to correctness and
termination of the modified algorithm, as they ensures refinement
of the ordering, rather than change. In particular, the compatible
leading terms identified by boundary vectors are indivisible by previous
leading terms; since refinement preserves them, each new polynomial
expands $\ideal{\lt\left(G\right)}$, and the Noetherian property
of a polynomial ring applies.

\section{Experimental results}

The current study implementation, written in Cython for the Sage computer
algebra system~\cite{sage48}, is available at

\begin{center}
\texttt{www.math.usm.edu/perry/Research/dynamic\_gb.pyx}
\par\end{center}

\noindent It is structured primarily by the following functions:
\begin{lyxlist}{MM}
\item [{\texttt{dynamic\_gb}}] is the control program, which invokes the
usual functions for a Buchberger algorithm (creation, pruning, and
selection of critical pairs using the Gebauer-M\"oller algorithm
and the sugar strategy, as well as computation and reduction of of
$S$-polynomials), as well as the following functions necessary for
a dynamic algorithm that uses the criteria of Disjoint Cones and Boundary
Vectors:

\begin{lyxlist}{MM}
\item [{\texttt{choose\_an\_ordering},}] which refines the ordering according
to the Hilbert function heuristic, and invokes:

\begin{lyxlist}{MM}
\item [{\texttt{possible\_lts},}] which applies the Boundary Vectors and
Divisibility criteria;
\item [{\texttt{feasible},}] which tries to extend the current linear program
with constraints corresponding to the preferred leading term; it also
applies the Disjoint Cones criterion, and invokes

\begin{lyxlist}{MM}
\item [{\texttt{monitor\_lts},}] which verifies that an ordering computed
by \texttt{feasible} preserves the previous choices of leading terms;
\end{lyxlist}
\end{lyxlist}
\item [{\texttt{boundary\_vectors},}] which computes an approximation $\Psi$
to the boundary vectors $\Omega$.
\end{lyxlist}
\end{lyxlist}
The \texttt{dynamic\_gb} function accepts the following options:
\begin{itemize}
\item \texttt{static}: boolean, \texttt{True} computes by the static method,
while \texttt{False} (the default) computes by the dynamic method;
\item \texttt{strategy}: one of \texttt{'sugar'}, \texttt{'normal'} (the
default), or \texttt{'mindeg'};
\item \texttt{weighted\_sugar}: boolean, \texttt{True} computes sugar according
to ordering, while \texttt{False} (the default) computes sugar according
to standard degree;
\item \texttt{use\_boundary\_vectors}: boolean, default is \texttt{True};
\item \texttt{use\_disjoint\_cones}: boolean, default is \texttt{True}.
\end{itemize}
At the present time, we are interested in structural data rather than
timings. To that end, experimental data must establish that the methods
proposed satisfy the stated aim of reducing the size and number of
linear programs constructed; in other words, the algorithm invokes
the refiner only when it has high certainty that it is needed. Evidence
for this would appear as the number of linear programs it does \emph{not}
construct, the number that it \emph{does}, and the ratio of one to
the other. We should also observe a relatively low number of failed
linear programs.

Table~\ref{tab: comparison dyn, static} summarizes the performance
of this implementation on several benchmarks. Its columns indicate:
\begin{itemize}
\item the name of a polynomial system tested;
\item the number of linear programs (i.e., potential leading terms)

\begin{itemize}
\item rejected using approximate boundary vectors,
\item rejected using disjoint cones,
\item solved while using the two new criteria,
\item solved while not using the new criteria, and
\item failed;
\end{itemize}
\item the ratio of the linear programs solved using the new criteria to
the number solved without them; and
\item the number of constraints in the final program, both using the new
criteria, and not using them.
\end{itemize}
Both to emphasize that the algorithm really is dynamic, and to compare
with Caboara's original results, Table~\ref{tab: comparison dyn, static}
also compares:
\begin{itemize}
\item the size of the Gr\"obner basis generated by the dynamic algorithm,
in terms of

\begin{itemize}
\item the number of polynomials computed, and
\item the number of terms appearing in the polynomials of the basis;
\end{itemize}
\item the size of the Gr\"obner basis generated by \textsc{Singular}'s
\texttt{std()} function with the grevlex ordering, in the same terms.
\end{itemize}
\begin{table}
\begin{centering}
\begin{sideways}
\begin{tabular}{|c|c|c|c|c|c|c|c|c|c|c|c|c|}
\multicolumn{1}{c|}{} & \multicolumn{8}{c}{linear programs} & \multicolumn{4}{c|}{final size}\tabularnewline
\multicolumn{1}{c|}{} & \multicolumn{2}{c|}{rejected by\ldots{}} & \multicolumn{2}{c|}{solved} &  &  & \multicolumn{2}{c|}{\#constraints} & \multicolumn{2}{c|}{dynamic} & \multicolumn{2}{c|}{static}\tabularnewline
\multicolumn{1}{c|}{system} & corners & disjoint & cor+dis & div only & $\frac{\mathrm{cor}+\mathrm{dis}}{\mathrm{div\, only}}$ & failed & cor+dis & div only & \#pols & \#terms & \#pols & \#terms\tabularnewline
\hline 
\hline 
Caboara 1 & 22 & 0 & 10 & 25 & 0.400 & 0 & 20 & 29 & 35 & 137 & 239 & 478\tabularnewline
\hline 
Caboara 2 & 19 & 0 & 20 & 10 & 2.222 & 10 & 26 & 17 & 21 & 42 & 553 & 1106\tabularnewline
\hline 
Caboara 4 & 64 & 0 & 10 & 20 & 0.500 & 2 & 26 & 30 & 9 & 20 & 13 & 23\tabularnewline
\hline 
Caboara 5 & 81 & 0 & 19 & 6 & 3.167 & 14 & 24 & 20 & 12 & 67 & 20 & 201\tabularnewline
\hline 
Caboara 6 & 15 & 0 & 7 & 8 & 0.875 & 1 & 16 & 16 & 7 & 15 & 7 & 15\tabularnewline
\hline 
Caboara 9 & 2 & 0 & 4 & 5 & 0.800 & 0 & 10 & 11 & 7 & 14 & 37 & 74\tabularnewline
\hline 
Cyclic-5 & 379 & 0 & 16 & 327 & 0.049 & 5 & 31 & 61 & 11 & 68 & 20 & 85\tabularnewline
\hline 
Cyclic-6 & 4,080 & 0 & 58 & 2800 & 0.021 & 43 & 56 & 250 & 20 & 129 & 45 & 199\tabularnewline
\hline 
Cyclic-7 & 134,158 & 12 & 145 & {*} & {*} & 108 & 147 & {*} & 63 & 1,049 & 209 & 1,134\tabularnewline
\hline 
Cyclic-5 hom. & 128 & 0 & 17 & 259 & 0.066 & 7 & 33 & 60 & 11 & 98 & 38 & 197\tabularnewline
\hline 
Cyclic-6 hom. & 1,460 & 0 & 25 & 1,233 & 0.020 & 16 & 39 & 303 & 33 & 476 & 99 & 580\tabularnewline
\hline 
Cyclic-7 hom. & 62,706 & 0 & 38 & {*} & {*} & 20 & 105 & {*} & 222 & 5,181 & 443 & 3,395\tabularnewline
\hline 
Katsura-6 & 439 & 5 & 37 & 108 & 0.343 & 18 & 53 & 68 & 22 & 54 & 22 & 54\tabularnewline
\hline 
Katsura-7 & 1,808 & 3 & 43 & 379 & 0.113 & 33 & 109 & 205 & 49 & 104 & 41 & 105\tabularnewline
\hline 
Katsura-6 hom. & 419 & 13 & 77 & 326 & 0.236 & 46 & 113 & 281 & 23 & 164 & 22 & 160\tabularnewline
\hline 
Katsura-7 hom. & 1,920 & 24 & 114 & 1,375 & 0.083 & 63 & 234 & 731 & 46 & 352 & 41 & 354\tabularnewline
\hline 
\end{tabular}
\end{sideways}
\par\end{centering}

{*}This system was terminated when using only the Divisibility Criterion,
as the linear program had acquired more than 1000 constraints.

\caption{\label{tab: comparison dyn, static}Dynamic algorithm with sugar strategy,
applying Divisibility Criterion after boundary vectors. Data for static
algorithm included for comparison.}
\end{table}
 The systems ``Caboara $i$'' correspond to the example systems
``Es $i$'' from~\cite{CaboaraDynAlg}, some of which came from
other sources; we do not repeat the details here. We verified by brute
force that the final result was a Gr\"obner basis.

The reader readily sees that the optimizations introduced in this
paper accomplish the stated goals. By itself, the method of boundary
vectors eliminates the majority of incompatible leading terms. In
the case of dense polynomial systems, it eliminates the vast majority.
This means that far, far fewer linear programs are constructed, and
those that are constructed have far fewer constraints than they would
otherwise. While the number of inconsistent linear programs the algorithm
attempted to solve (the ``failed'' column) may seem high in proportion
to the number of programs it did solve (``solved''), this pales
in comparison to how many it would have attempted without the use
of boundary vectors; a glance at the ``div only'' column, which
consists of monomials that were eliminated only by the Divisibility
Criterion, should allay any such concerns.
\begin{rem}
In two cases, the new criteria led the algorithm to compute \emph{more}
linear programs than if it had used the Divisibility Criterion alone.
There are two reasons for this.
\begin{itemize}
\item One of the input systems (Caboara 2) consists exclusively of inhomogeneous
binomials with many divisible terms. This setting favors the Divisibility
Criterion. The other system (Caboara 5) also contains many divisible
terms.
\item For both systems, the sample set of boundary vectors wrongly eliminates
compatible monomials that should be kept. While this suggests that
the current strategy of selecting a sample set of boundary vectors
leaves much to be desired, the dynamic algorithm computes a smaller
basis than the static, and with fewer $S$-poly\-nomials, even in
these cases.
\end{itemize}
\end{rem}
It should not startle the reader that the dynamic algorithm performs
poorly on the homogeneous Katsura-$n$ systems, as Caboara had already
reported this. All the same, boundary vectors and disjoint cones minimize
the cost of the dynamic approach.

Many of our Gr\"obner bases have different sizes from those originally
reported by Caboara. There are several likely causes:
\begin{itemize}
\item The original report had typographical errors. We have verified and
corrected this in some cases, but some results continue to differ.
\item The static ordering used here may order the variables differently
from~\cite{CaboaraDynAlg}, which did not documented this detail.
For example, Caboara reports a basis of only 318 polynomials for Caboara~2
when using the static algorithm with grevlex, but Sage (using \textsc{Singular})
finds 553.
\item Several choices of leading term can have the same tentative Hilbert
function, but one of the choices is in fact better than the others
in the long run. The original implementation may have chosen differently
from this one. In particular~\cite{CaboaraDynAlg} gave a special
treatment to the input polynomials, considering the possible leading
term choices for \emph{all} of them simultaneously, and not sequentially,
one-by-one.
\end{itemize}
We conclude this section with a word on complexity. While the worst
case time complexity of the simplex algorithm is exponential~\cite{KleeMinty72},
on average it outperforms algorithms with polynomial time complexity.
The computation of boundary vectors increases the number of invocations
of simplex, but the large reduction in the number of monomials considered
more than compensates for this. The space requirements are negligible,
as we need only $2n$ boundary vectors at any one time, and the reduction
in the number and size of the linear programs means the algorithm
needs to remember only a very few disjoint cones. Considering how
rarely the disjoint cones are useful, it might be worthwhile not to
implement them at all, but we have not observed them to be a heavy
burden at the current time.

\section{Conclusion, future work}

We set out to reduce the size and number of linear programs used by
a dynamic algorithm to compute a Gr\"obner basis. Geometrical intuition
led us to two methods that work effectively and efficiently. While
the effect with the systems tested by Caboara was only moderate, and
in some cases counterproductive, the story was different with the
dense benchmark systems. In these cases, the number of refinements
approached insignificance, and we continued to achieve good results.
The final Gr\"obner basis was always of a size significantly smaller
than grevlex; only a few refinements were required for the algorithm
to be very effective.

A significant restraint imposed by many computer algebra systems is
that a term ordering be defined by integers. As Sage is among these,
this has required us to solve not merely linear programs, but pure
integer programs. Integer programming is much more intensive than
linear programming, and its effect was a real drag on some systems.
There is no theoretical need for this; we plan to look for ways to
eliminate this requirement, or at least mitigate it.

An obvious next step is to study various ambiguities in this approach.
This includes traditional questions, such as the effect of the selection
strategy of critical pairs, and also newer questions, such as generating
the sample set of boundary vectors. We followed Caboara's approach
and used the sugar strategy. We expect the normal strategy~\cite{Buchberger85}
to be useful only rarely, while signature-based strategies~\cite{EderPerrySigBasedAlgorithms}\cite{Fau02Corrected},
which eliminate useless pairs using information contained in the leading
terms of a module representation, could combine with the dynamic algorithm
to expand significantly the frontiers of the computation of Gr\"obner
bases, and we are working towards such an approach.

Another optimization would be to use sparse matrix techniques to improve
the efficiency of polynomial reduction in a Gr\"obner basis algorithm,
\emph{a la} F4. Combining this with a dynamic algorithm is comparable
to allowing some column swaps in the Macaulay matrix, something that
is ordinarily impossible in the middle of a Gr\"obner basis computation.

\thanks{The authors would like to thank Nathann Cohen for some stimulating
conversations, and his assistance with Sage's linear programming facilities.}

\section*{Appendix}

Here we give additional information on one run of the inhomogeneous
Cyclic-6 system. Computing a Gr\"obner basis for this ideal with
the standard sugar strategy requires more $S$-poly\-nomials than
doing so for the ideal of the homogenized system; in our implementation,
the number of $S$-poly\-nomials is roughly equal to computing the
basis using a static approach. (In general, the dynamic approach takes
fewer $S$-poly\-nomials, and the weighted sugar strategy is more
efficient on this ideal.) Information on timings was obtained using
Sage's profiler:
\begin{itemize}
\item for the static run, we used
\end{itemize}
\begin{center}
\texttt{\%prun B = dynamic\_gb(F,static=True,strategy='sugar')};
\par\end{center}
\begin{itemize}
\item for the dynamic run, we used
\end{itemize}
\begin{center}
\texttt{\%prun B = dynamic\_gb(F,static=False,strategy='sugar')}.
\par\end{center}

We obtained structural data by keeping statistics during a sample
run of the program.

While Cython translates Python instructions to C code, then compiles
the result, the resulting binary code relies on both the Python runtime
and Python data structures; it just works with them from compiled
code. Timings will reflect this; we include them merely to reassure
the reader that this approach shows promise.

Our implementation of the dynamic algorithm uses 377 $S$-poly\-nomials
to compute a Gr\"obner basis of 20 polynomials, with 250 reductions
to zero. At termination, $\rejects$ contained 43 sets of constraints,
which were never used to eliminate linear programs. As for timings,
one execution of the current implementation took 25.45 seconds. The
profiler identified the following functions as being the most expensive.
\begin{itemize}
\item Roughly half the time, 11.193 seconds, was spent reducing polynomials.
(All timings of functions are cumulative: that is, time spent in this
function and any other functions it invokes.) This is due to our having
to implement manually a routine that would reduce polynomials and
compute the resulting sugar. A lot of interpreted code is involved
in this routine.
\item Another third of the time, 8.581 seconds, was spent updating the critical
pairs using the Gebauer-M\"oller update algorithm. Much of this time
(3.248 seconds) was spent computing the lcm of critical pairs.
\end{itemize}
The majority of time (5/6) was spent on functions that are standard
in the Buchberger algorithm! Ordinarily, we would not expect an implementation
to spend that much time reducing polynomials and updating critical
pairs; we are observing a penalty from the Python runtime and data
structures. The size of this penalty naturally varies throughout the
functions, but this gives the reader an idea of how large it can be.

Other expensive functions include:
\begin{itemize}
\item About 4.6 seconds were spent in \texttt{choose\_an\_ordering}. Much
of this time (2.4 seconds) was spent sorting compatible leading terms
according to the Hilbert heuristic.
\item The 238 invocations of Sage's \texttt{hilbert\_series} and \texttt{hilbert\_polynomial}
took 1.5 seconds and 0.8 seconds, respectively. As Sage uses \textsc{Singular}
for this, some of this penalty is due to compiled code, but \textsc{Singular}'s
implementation of these functions is not state-of-the-art; techniques
derived from~\cite{Bigatti97} and~\cite{RouneHilbert2010} would
compute the Hilbert polynomial incrementally and quickly.
\item The 57 invocations of \texttt{feasible} consumed roughly 1.2 seconds.
This includes solving both continuous and pure integer programs, and
suffers the same penalty from interpreted code as other functions.
\end{itemize}
It is worth pointing out what is \emph{not} present on this list:
\begin{itemize}
\item Roughly one quarter of a second was spent in \texttt{monitor\_lts}.
\item The \texttt{boundary\_vectors} function took roughly one and a half
tenths of a second (.016).
\item Less than one tenth of a second was spent applying the Boundary Vectors
Criterion in \texttt{possible\_lts}.
\end{itemize}
If we remove the use of the Boundary Vector and Disjoint Cones criteria,
the relative efficiency of \texttt{feasible} evaporates; the following
invocation illustrates this vividly:

\begin{center}
\texttt{}%
\begin{minipage}[t]{0.8\columnwidth}%
\texttt{\%prun B = dynamic\_gb(F, strategy='sugar',}

\texttt{~~~~~use\_boundary\_vectors=False,}

\texttt{~~~~~use\_disjoint\_cones=False)}%
\end{minipage}
\par\end{center}

\noindent Not only is the Divisibility Criterion unable to stop us
from solving 2,820 linear programs, but the programs themselves grow
to a size of 247 constraints. At~41 seconds, \texttt{feasible} takes
nearly twice as long as entire the computation when using the new
criteria! This inability to eliminate useless terms effect cascades;
\texttt{hilbert\_polynomial} and \texttt{hilbert\_series} are invoked
3,506 times, taking 12.6 and 28.6 seconds, respectively; \texttt{choose\_an\_ordering}
jumps to 92.7 seconds, with 44 seconds wasted in applying the heuristic.
By contrast, the timings for reduction of polynomials and sorting
of critical pairs grow much, much less, to 20 seconds and 12.9 seconds,
respectively. (The increase in reduction corresponds to an increase
in the number of $S$-poly\-nomials, probably due to a different
choice of leading term at some point -- recall that the approximation
of boundary vectors excludes some choices.) The optimizations presented
here really do remove one of the major bottlenecks of this method.

We conclude by noting that when we use the Disjoint Cones criterion
alone, 1,297 invalid monomials are eliminated, but the number of constraints
in the final linear program increases to 250; this compares to 4,080
monomials that Boundary Vectors alone eliminates, with 56 constraints
in the final linear program.

\bibliographystyle{spmpsci}
\bibliography{/home/perry/common/Research/researchbibliography}

\end{document}